\documentclass[12pt,letterpaper]{article}
\usepackage{graphicx}
\usepackage{amssymb,amsmath,amsthm,latexsym,tikz,pgf}
\usepackage{mathrsfs} 
\usepackage{amssymb}
\usepackage{epsfig}
\usepackage{epstopdf}
\usetikzlibrary{arrows.meta}
\usepackage[all]{xy}
\usepackage{subcaption}

\usepackage{float}

\newtheorem{thm}{Theorem}[section]

\newtheorem{lem}[thm]{Lemma}
\newtheorem{prop}[thm]{Proposition}

\newtheorem{defn}{Definition}[section]

\newcommand{\T}{\mathcal{T}}
\renewcommand{\L}{\mathcal{L}}
\renewcommand{\H}{\mathcal{H}}

\newcommand{\F}{\mathcal{F}}
\newcommand{\G}{\mathcal{G}}

\renewcommand{\phi}{\varphi}
\newcommand{\RR}{\mathbb{R}}
\newcommand{\ZZ}{\mathbb{Z}}
\newcommand{\eps}{\varepsilon}

\begin{document}

\title{Polynomial Dynamical Systems, 
Reaction Networks, and
Toric Differential Inclusions
}

\author{
Gheorghe Craciun\\ 
Department of Mathematics and \\ 
Department of Biomolecular Chemistry\\
University of Wisconsin-Madison\\
e-mail: \texttt{craciun@math.wisc.edu}\\
\\
\\
}

\date{\today}

\maketitle

\begin{abstract}
\noindent 
Some of the most common mathematical models 
in biology, chemistry, physics, and engineering, are polynomial dynamical systems, i.e., systems of differential equations with polynomial right-hand sides.
Inspired by notions and results that have been developed for the \mbox{analysis} of reaction networks in biochemistry and chemical engineering, we show that any polynomial dynamical system on the positive orthant $\RR^n_{> 0}$ can be regarded as being generated by an oriented graph embedded in $\RR^n$, called {\em Euclidean embedded graph}. 
This allows us to recast  key conjectures about reaction network models (such as the Global Attractor Conjecture, or the Persistence Conjecture) into more general versions about some important classes of polynomial dynamical systems. Then, we introduce {\em toric differential inclusions}, which are piecewise constant autonomous dynamical systems with a remarkable geometric structure. We show that if a Euclidean embedded graph $\G$ has some reversibility properties, then any polynomial dynamical system generated by $\G$ can be embedded into a toric differential inclusion. We discuss  how this embedding suggests an approach for the proof of the Global Attractor Conjecture and Persistence Conjecture. 
%


\end{abstract}

\section{Introduction}

Many mathematical models in biology, chemistry, physics, and engineering are given by polynomial dynamical systems, or more generally, power-law dynamical systems \cite{CNP, TDS, ctf06, Craciun_Tran, Banaji_2013, mf72,  Feinberg_1979, Feinberg_1987, gunawardena, Horn_Jackson, Savageau_Voit, Yu_Craciun}. Almost always, these can  be interpreted as population dynamics models, where the variables of interest are positive.

Any autonomous polynomial dynamical system (i.e., system of differential equations with polynomial right-hand side)  on the strictly positive orthant $\RR^n_{> 0}$ can be represented as 
\begin{equation}\label{polynomial_1}
\frac{dx}{dt} = \sum_{i = 1}^m x^{s_i} v_i 
\end{equation}
where $x = (x_1, ..., x_n) \in \RR^n_{> 0}$, $s_1, ..., s_m$ are some vectors in $\ZZ^n_{\ge 0}$ called {\em exponent vectors},  $x^{s_i}$ denotes the monomial $x_1^{s_{i1}} x_2^{s_{i2}} ... x_n^{s_{in}}$, and $v_1, ..., v_m$ are  vectors in $\RR^n$. A {\em solution} of (\ref{polynomial_1}) is a function $x : I \to  \RR^n_{> 0}$ that satisfies (\ref{polynomial_1}), where $I$ is an interval in $\RR$. 


Note  that, since the coordinates $x_1, ..., x_n$ are positive, the  monomials $x^{s_i}$ are well-defined even if the coordinates of the exponent vectors $s_i$ are arbitrary real numbers (i.e., $s_1, ..., s_m$ are not necessarily in $\ZZ_{\ge 0}^n$). In that case we say that~(\ref{polynomial_1}) is a {\em power-law dynamical system}. The approaches and results discussed in this paper apply not only to polynomial dynamical systems, but also to power-law dynamical systems. In this paper, whenever we say ``polynomial dynamical system", we mean ``polynomial or power law dynamical system". 

In many applications  there are also some positive parameter values in these systems, which may be difficult to estimate accurately (such as reaction rate constants in biochemistry and chemical engineering, or interaction rates in epidemiology and ecology). Then the dynamical system of interest may have the form 
\begin{equation}\label{polynomial_k}
\frac{dx}{dt} = \sum_{i = 1}^m k_i x^{s_i} v_i 
\end{equation}
where $k_1, ..., k_m$ are some positive constants. In this case, we may want to know if some properties of the solutions of the system (\ref{polynomial_k}) may hold {\em for all} choices of positive parameters $k_i$. 

In other cases, the interaction network we need to model is part of a larger network that contains variables or ``external factors" that influence our system, but are not contained in our system. In that case we cannot use an autonomous dynamical system as a model, but we may be able to use a nonautonomous dynamical system of the form
\begin{equation}\label{polynomial_nonauto}
\frac{dx}{dt} = \sum_{i = 1}^m k_i(t) x^{s_i} v_i 
\end{equation}
where the functions $k_i$ are positive and uniformly bounded, i.e., there exists some $\eps>0$ such that $\eps \le k_i(t) \le \frac{1}{\eps} $ for all $t$.
We will refer to models of the form (\ref{polynomial_nonauto}) as {\em variable-$k$ polynomial dynamical systems}.

In applications, there is great interest in understanding the {\em global stability} and {\em persistence} properties of dynamical systems of the form (\ref{polynomial_k}) and (\ref{polynomial_nonauto}). For example, a natural question is the following: for what systems (\ref{polynomial_nonauto}) is it true that all solutions have a positive lower bound for all $t>0$ (i.e., no variable ``goes extinct"), irrespective of the choices of uniformly bounded external factors $k_i(t)$? 

In this paper we  describe an approach for analyzing such problems, even in the presence of unknown parameters (as in (\ref{polynomial_k})) or external factors (as in (\ref{polynomial_nonauto})). In Section 2 we show that any polynomial dynamical system can be regarded as being generated by some ``Euclidean embedded graph" (also called ``E-graph"). In Section 3 we introduce the notion of ``toric differential inclusion", and we show that if an E-graph is reversible, then any (variable-$k$) polynomial dynamical system generated by it can be embedded into a toric differential inclusion. Then, in Section 4 we show that such an embedding still exist even if the reversibility restriction is relaxed significantly. In Section 5 we discuss how these embeddings may greatly simplify the analysis of some properties of polynomial dynamical systems.

\section{Euclidean embedded graphs}

A {\em Euclidean embedded graph} (or {\em E-graph}) is a finite oriented graph $\G = (V,E)$ whose vertices are labeled by distinct elements of $\RR^n$ for some $n \ge 1$. With an abuse of notation, we identify the set $V$ with the set of vertex labels, i.e., we assume that $V \subset \RR^n$.
Moreover, we associate to each edge $e = (s,t) \in E$  its {\em edge vector} $v(e) = t-s$. Also, we define its {\em source vertex} to be $s(e) = s$, and its  {\em target vertex} to be $t(e) = t$. 

\smallskip

Given an Euclidean embedded graph $\G = (V,E)$, the {\em polynomial dynamical  systems generated by $\G$} are the dynamical systems on $\RR^n_{> 0}$ given by 
\begin{equation}\label{polynomial_G}
\frac{dx}{dt} = \sum_{e \in E} k_e x^{s(e)} v(e) 
\end{equation}
for some positive constants $k_e$. Note that if $V \subset \ZZ^n_{\ge 0}$ then (\ref{polynomial_G}) is just {\em mass-action kinetics} for a chemical reaction network represented by $\G$, i.e, one where there is a reaction of the form $s(e) \to t(e)$ for each edge $e \in E$. (Informally speaking, for mass-action systems the rate of each reaction is proportional to the product of the concentrations of all its reactants, i.e., the rate of the reaction $s(e) \to t(e)$ is proportional to $x^{s(e)}$; see~\cite{CNP, Feinberg_1979, gunawardena, Savageau_Voit} for more details.) 

More generally, the {\em variable-k polynomial dynamical systems generated by $\G$} are the (nonautonomous) dynamical systems on $\RR^n_{> 0}$ given by 
\begin{equation}\label{polynomial_G_nonaut}
\frac{dx}{dt} = \sum_{e \in E} k_e(t) x^{s(e)} v(e) 
\end{equation}
such that there exists some $\eps>0$ for which we have $\eps \le k_e(t) \le \frac{1}{\eps}$ for all $e \in E$ and for all  $t$.

\smallskip

Let us note that for any variable-$k$ polynomial dynamical system (\ref{polynomial_nonauto})
we can construct an E-graph $\G$ that generates it, and $\G$ is {\em not} unique. Assuming that the ordered pairs $(s_1, v_1), (s_2, v_2), ..., (s_m, v_m)$ are distinct, the simplest way to construct such a $\G$ is to choose the set of vertices 
$$
V = \{ s_i \ | \ i=1,...,m \} \cup \{ s_i+v_i \ | \ i=1,...,m \}, 
$$
and the set of edges
$$
E = \{ (s_i, s_i+v_i) \ | \ i=1,...,m \}.
$$
If we want to obtain a different  E-graph that generates~(\ref{polynomial_nonauto}),  we can, for example, write one of the vectors $v_i$ as a positive linear combination of two different nonzero vectors, and use these new vectors to obtain a graph with  $m+1$ edges that also generates (\ref{polynomial_nonauto}).

\noindent

\medskip

We will use E-graphs in order to try to identify the polynomial dynamical systems that are known to have (or are conjectured to have) important dynamical properties, such as persistence, permanence, and global stability. For this purpose, we first define some special kinds of E-graphs (namely, reversible and weakly reversible E-graphs), and then we focus our attention on polynomial dynamical systems that are generated by these special kinds of graphs.

\medskip

We say that an E-graph $\G = (V,E)$ is {\em reversible} if for any edge $(s,t)$ in $E$ the reverse edge $(t,s)$ also belongs to $E$. Also we say that $\G$ is {\em weakly reversible} if any edge $(s,t)$ is part of an oriented cycle in $\G$, or equivalently, any connected component of $\G$ is strongly connected (a {\em strongly connected} directed graph is one where there exists a directed path between every pair of vertices). For example, graph in Fig.~\ref{fig:1}$(e)$ is reversible, and the graphs in Fig.~\ref{fig:1}$(c)$ and $(d)$ are weakly reversible. Of course, every reversible E-graph is also weakly reversible.

We say that a polynomial dynamical system is {\em reversible} if there exists some reversible E-graph that generates it, and, we say that a polynomial dynamical system is {\em weakly reversible} if there exists some weakly reversible E-graph that generates it. Analogously, we define variable-$k$ reversible and weakly reversible polynomial dynamical systems.

\begin{figure}[h!]
  \begin{center}
    \begin{tabular}{c}
    \hskip-1.4cm
      $(a)$ 
      \includegraphics[width=2.01in]{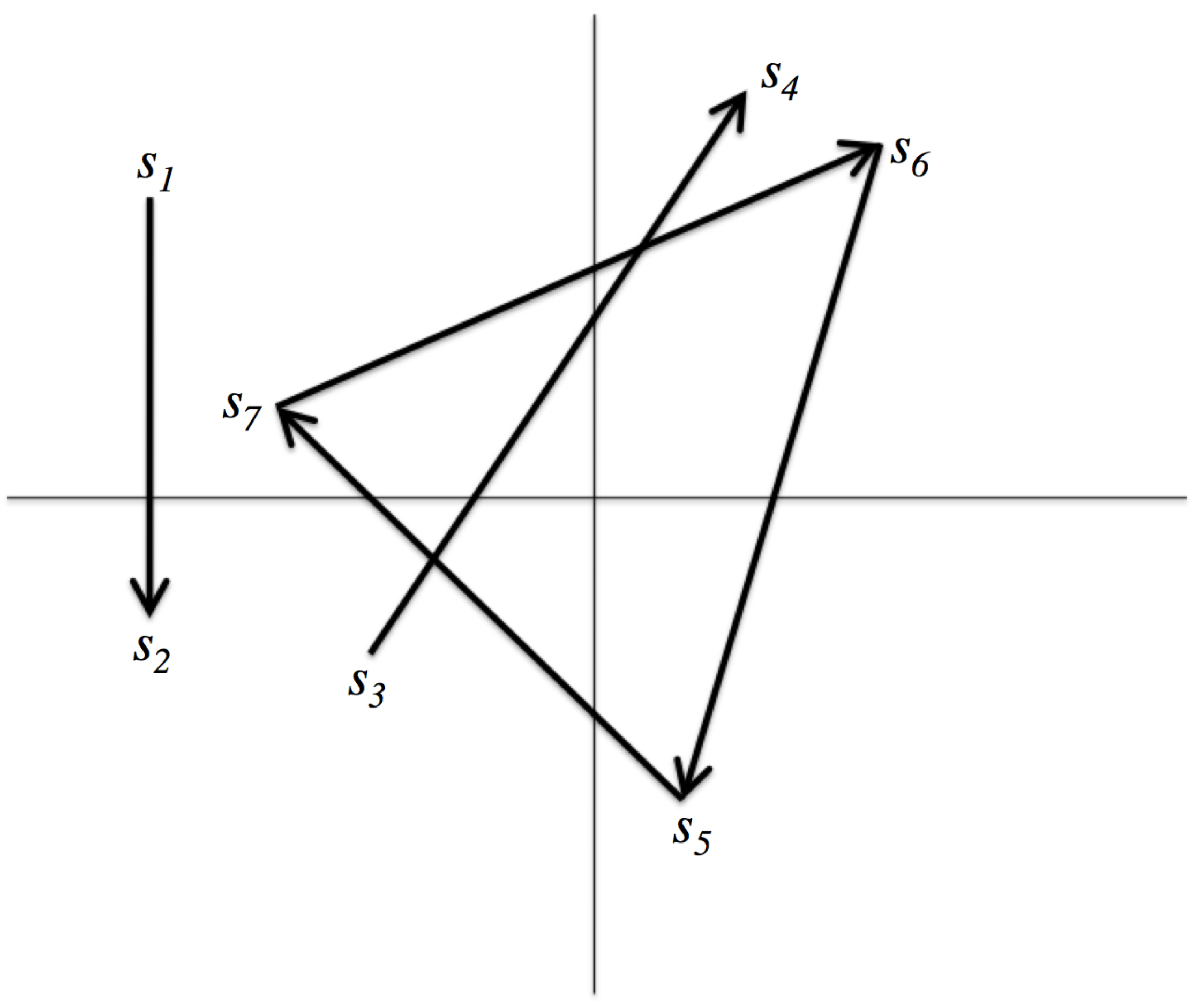}
      $(b)$ 
      \ \ \includegraphics[width=2.01in]{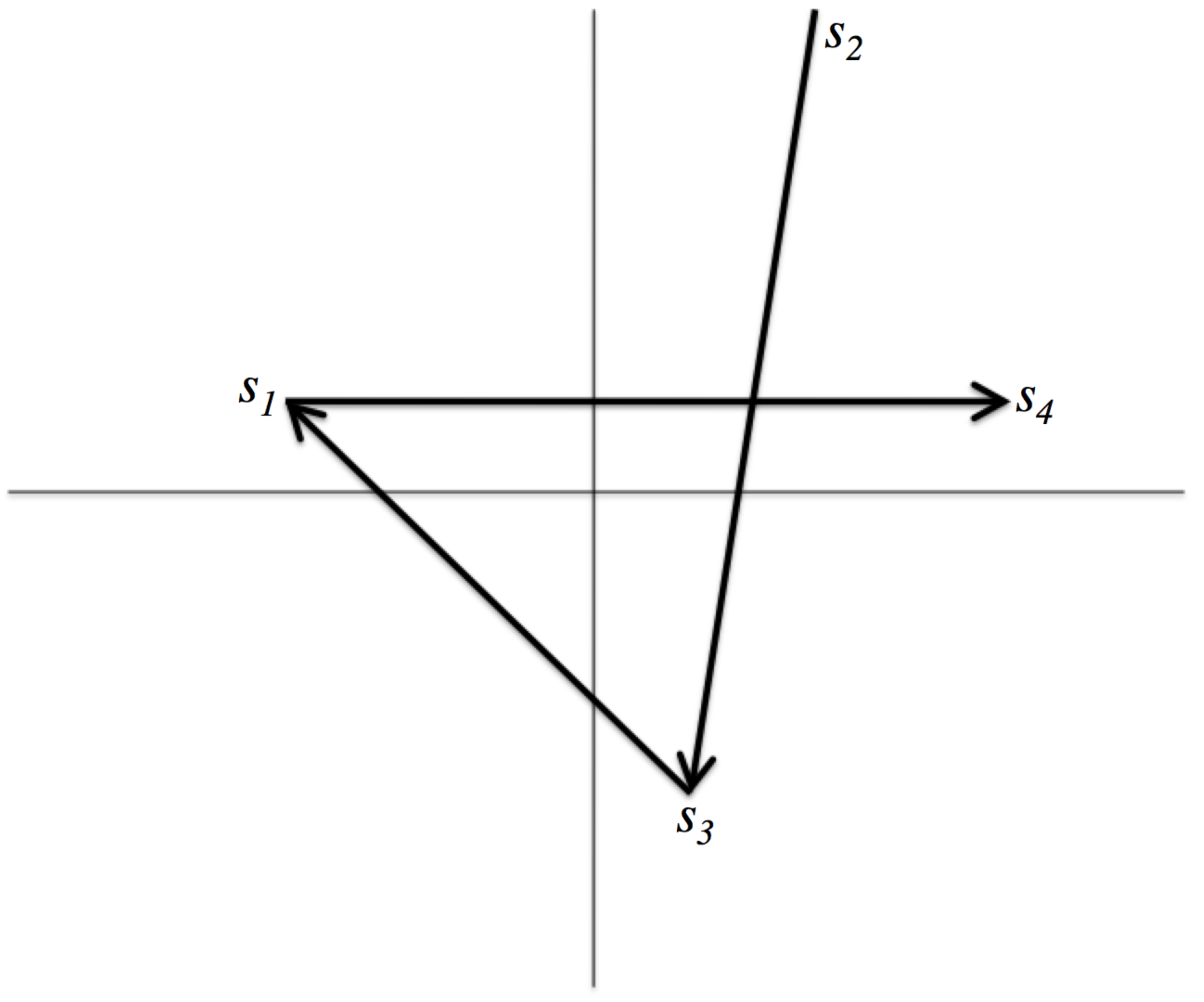}
      $(c)$ 
      \ \ \includegraphics[width=2.01in]{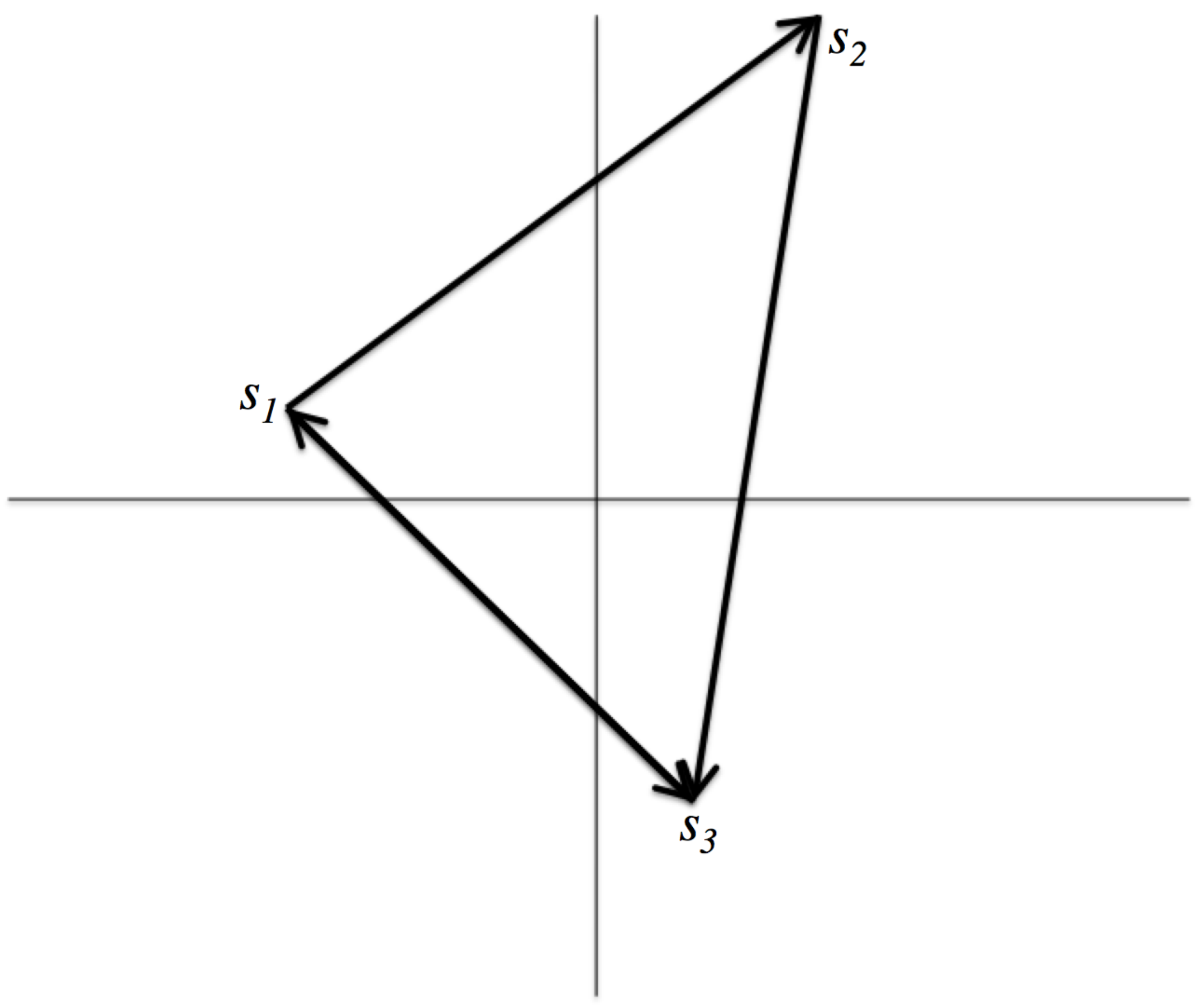} \\
    \hskip-1.4cm
      $(d)$ 
      \includegraphics[width=2.01in]{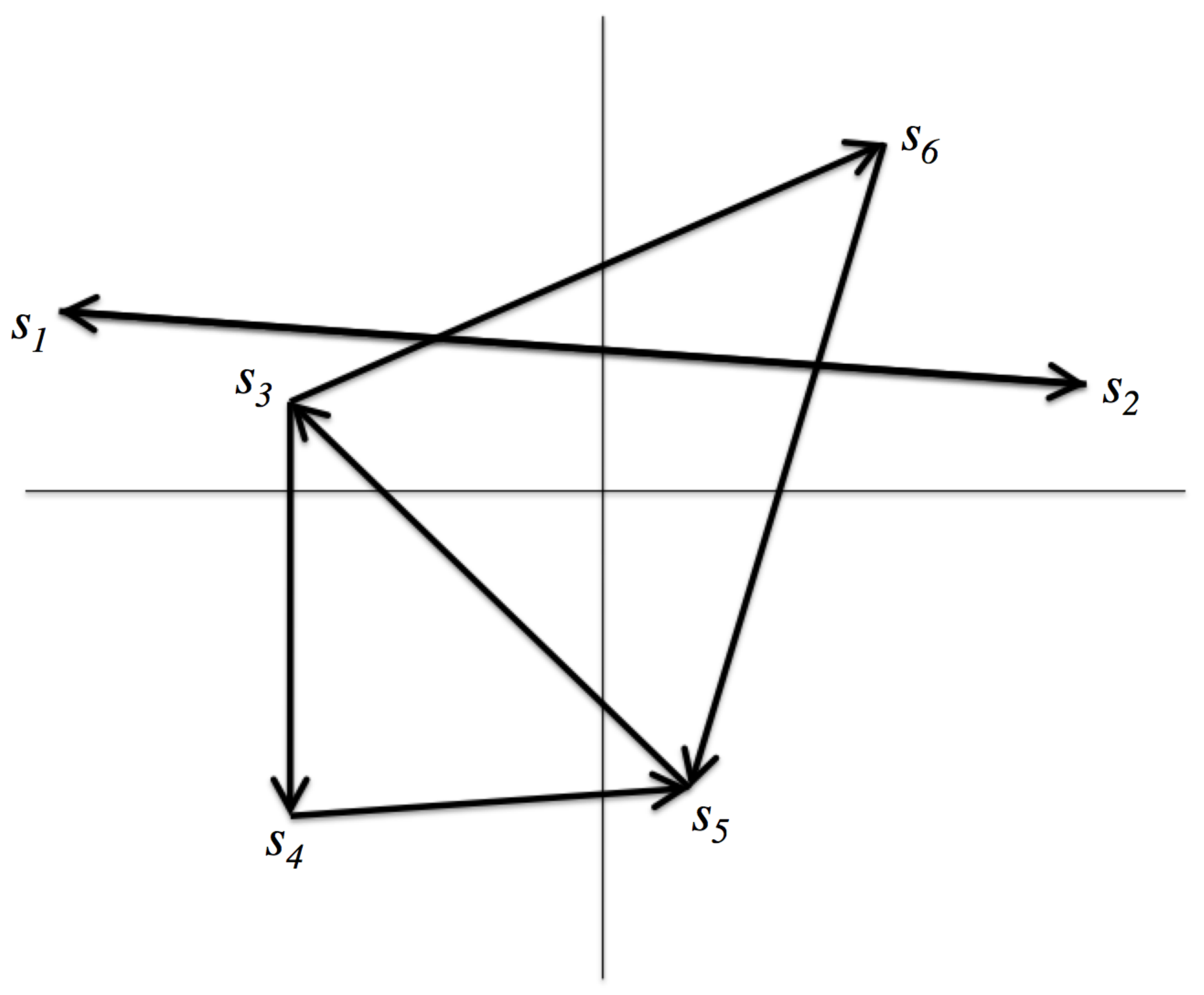}
      \ \ \ \ \ \ \ \ \ 
      $(e)$ 
      \includegraphics[width=2.01in]{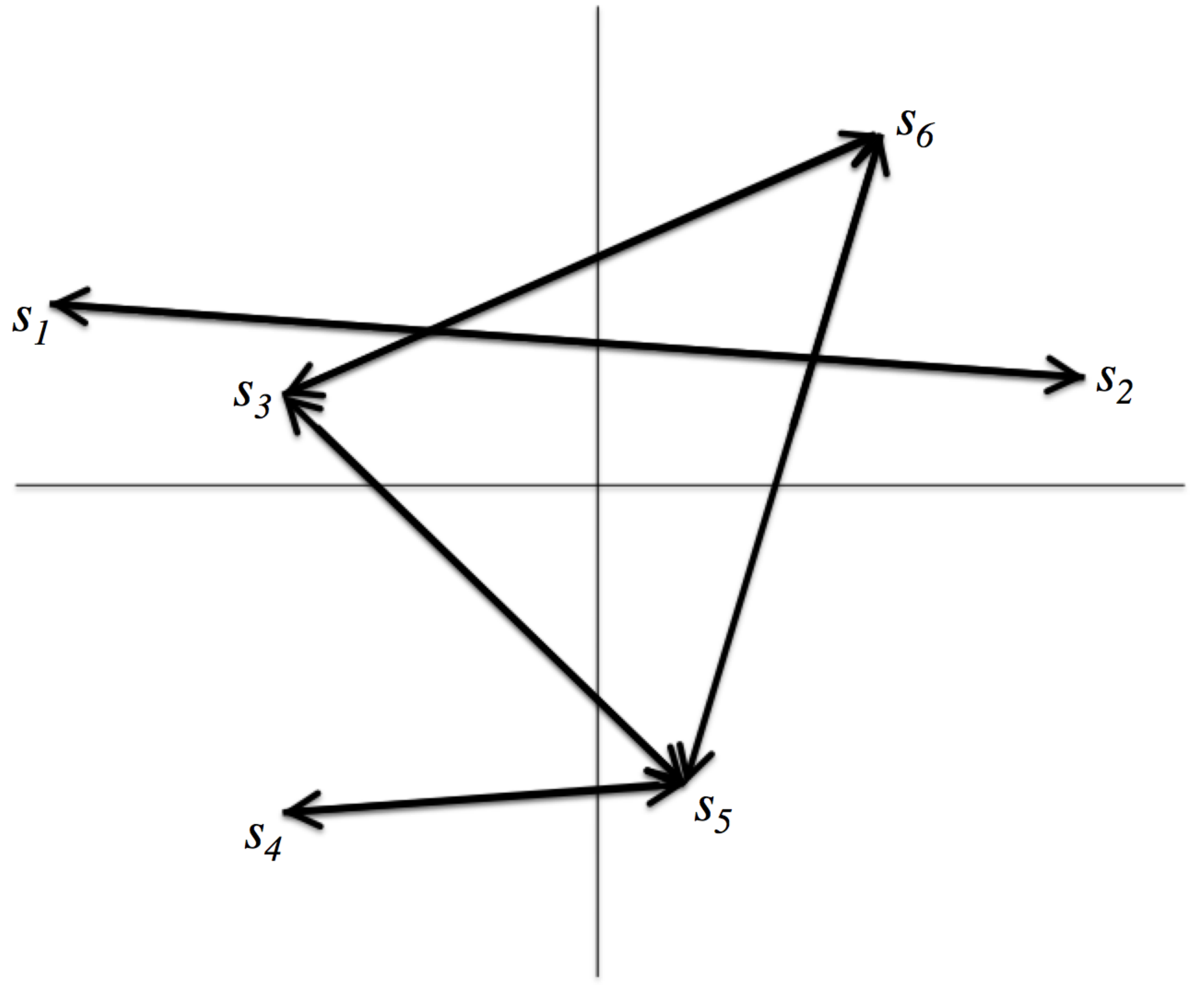}
    \end{tabular}
  \end{center}
  \caption{\label{fig:1}{\it Five examples of E-graphs in $\RR^2$. For each graph, while we {\em do} assume that the points $s_1, ..., s_m\in\RR^2$ are distinct, note that we do {\em not} assume that the line segments (i.e., arrows) representing the vectors $s_j - s_i$ are disjoint. The  graphs $(d)$ and $(e)$ are  {\em weakly reversible}, and the  graph $(e)$ is actually {\em reversible}.
Although the graph $(b)$ is not weakly reversible, it generates dynamical systems~(\ref{polynomial_G}) which {\em can also be represented by a weakly reversible graph} (e.g., the graph $(c)$), because the vector $s_4 - s_1$ is a positive linear combination of the vectors $s_2 - s_1$ and $s_3 - s_1$, so the term corresponding to the edge $(s_1,s_4)$ can be replaced by two terms, one corresponding to the edge $(s_1,s_2)$, and the other corresponding to the edge $(s_1,s_3)$. On the other hand, the dynamical systems generated by the graph $(a)$ {\em cannot} be generated by a weakly reversible graph~\cite{ABCM_2018}.}}
\end{figure}

\bigskip

\noindent 
{\bf Example 1.} Consider the dynamical system given by
\begin{eqnarray}\label{example_1}
\frac{dx_1}{dt}  & =  & -2k_1(t) x_1^2 + 2k_2(t) x_2 \\ \nonumber
\frac{dx_2}{dt}  & =  &   k_1(t) x_1^2  -  k_2(t) x_2
\end{eqnarray}
for some  functions $k_i(t)$ with $\eps < k_i(t) < \frac{1}{\eps}$ for all $t$.

This system can be written in vector form, as follows:
\begin{eqnarray}\label{example_1(2)}
\frac{dx}{dt} = k_1(t) x_1^2 \binom{-2}{1} + k_2(t) x_2 \binom{2}{-1},
\end{eqnarray}
where $x = \binom{x_1}{x_2}$. In turn, this can be written in the form (\ref{polynomial_nonauto}), as follows:
\begin{eqnarray}\label{example_1(4)}
\frac{dx}{dt} = k_1(t) x^{s_1} \binom{-2}{1} + k_2(t) x^{s_2} \binom{2}{-1},
\end{eqnarray}
where $s_1 = \binom{2}{0}$ and $s_2 = \binom{0}{1}$. Then, the simplest E-graph $\G$ that generates the dynamical system (\ref{example_1}) has two edges, one edge going from $s_1$ to $s_1' := s_1 + \binom{-2}{1}$, and the other edge going from $s_2$ to $s_2' := s_2 + \binom{2}{-1}$. But, note that we happen to  have $s_1' = s_2$ and $s_2' = s_1$, so the graph $\G$ actually has only two vertices, and is reversible. The graph $\G$ is shown in Fig.~\ref{fig:3}$(a)$ in Section 3. We will return to this kind of example in section 3, when we will see that the dynamics of this 
system can be understood by embedding it into a special kind of differential inclusion.

\subsection{Persistent, permanenent, and globally stable polynomial dynamical systems}

%
We say that a \text{variable-$k$} polynomial dynamical system in $\RR^n_{> 0}$ is {\em persistent} if, for any solution $x(t)$ defined on an interval $I$ that contains $t=0$, there exists some $\eps_0>0$ such that 
$$
x_i(t) > \eps_0  \ \mathrm{\ for \ all \ }  i \in \{ 1,...,n\}  \mathrm{\ and \ for \ all \ } t \in I \cap [0,\infty).
$$
In other words, the system is called persistent if, for any solution $x(t)$ with positive initial condition, there exists a positive lower bound for all the variables $x_i(t)$, and for all the future times at which the solution is defined (we cannot say that there exists a positive lower bound for all $t>0$ for a technical reason: some solutions may blow up in finite time). Informally, persistence means that ``no variable goes extinct".

To define the permanence property we first need to point out that polynomial dynamical systems  have some special invariant spaces. If a (variable-$k$) polynomial dynamical system is given by (\ref{polynomial_G}) or (\ref{polynomial_G_nonaut}), then its {\em  edge space} $S$ is the linear span of the set of edge vectors $\{v(e) \, | \, e \in E\}$. Then we define its {\em affine invariant sets} to be the sets of the form 
$$
(x_0+S) \cap \RR^n_{> 0},  \ \mathrm{\ for \ some \ } x_0 \in \RR^n_{> 0}.
$$
These are indeed invariant spaces for  solutions of  (\ref{polynomial_G}) or (\ref{polynomial_G_nonaut}) on the domain $\RR^n_{> 0}$, because all the vectors that appear on the right-hand side of these equations are linear combinations of $v_1,...,v_m$, so are contained in $S$. 

We say that a (variable-$k$) polynomial dynamical system on $\RR^n_{> 0}$ is {\em permanent} if, for each affine invariant set $S_{x_0} = (x_0+S) \cap \RR^n_{> 0}$ there exists a compact set $K_{x_0} \subset S_{x_0}$ such that any solution $x(t)$ with $x(0) \in S_{x_0}$ can be extended for all $t>0$, and there exists some $t_0>0$ such that $x(t) \in K_{x_0}$ for all $t>t_0$.
It follows that, for a permanent system, all solutions that start in $S_{x_0}$ can be extended for all $t>0$, and, for large enough $t$, they are bounded above and below by some positive constants (and these positive upper and lower bounds do not depend on the initial condition, while for persistent systems the lower bound $\eps$ may depend on the initial condition). In particular, permanence implies persistence. 

Also, we say that a polynomial dynamical system (\ref{polynomial_k}) {\em has a  globally attracting point} within the affine invariant set $S_{x_0}$ if there exists a point $\bar x_0 \in S_{x_0}$ such that any solution $x(t)$ with $x(0) \in S_{x_0}$ can be extended for all $t>0$ and we have 
$$
\lim_{t \to \infty} x(t) = \bar x_0.
$$

Also, we say that a polynomial dynamical system (\ref{polynomial_k}) is {\em vertex balanced} if there exists an E-graph $\G = (V,E)$ that generates our system as in (\ref{polynomial_G}), and there exists a point $\bar x \in \RR^n_{> 0}$ such that for any vertex $s \in V$ we have 
\begin{equation}\label{VBE}
\sum_{e = (s,s') \in E} k_e \bar x^{s} =  \sum_{e = (s',s) \in E} k_e \bar x^{s'}. 
\end{equation}
In other words, if we think of the positive number $k_{e}x^{s(e)}$ as the rate of a flow along the edge $e$ (i.e., a flow from the vertex $s(e)$ to the vertex $t(e)$), then condition (\ref{VBE}) says that, if $x=\bar x$, then at each vertex of the graph $G$, the sum of all the incoming flows equals the sum of all outgoing flows. 

The notion of ``vertex balanced polynomial dynamical system" (or ``vertex balanced power-law dynamical system") is a natural generalization of the notion of ``toric dynamical system", which in turn was a reformulation of the notion of ``complex balanced mass-action system", which was introduced by Fritz Horn and Roy Jackson in their seminal work on models of reaction networks with mass action kinetics ~\cite{Horn_Jackson}. For more details see~\cite{Craciun_GAC, TDS, Feinberg_1979, gunawardena}. This notion ultimately originates in the work of Boltzmann~\cite{Boltzmann_1887, Boltzmann_1896}. For some recent connections between  polynomial dynamical systems, reaction networks, and the Bolzmann equation, see~\cite{Craciun_Tran}.

%
%

\subsection{Open problems}

We can now formulate the following conjectures, inspired by analogous conjectures that have been formulated for mass-action systems \cite{CNP, TDS, Horn_Jackson, Horn_1974}, and are widely regarded as the key open problems in this field.

\medskip

{\bf Global Attractor Conjecture.} Any vertex balanced polynomial dynamical system has a globally attracting point within any affine invariant set.

\medskip

{\bf Extended Persistence Conjecture.} Any variable-$k$ weakly reversible polynomial dynamical system is persistent. 

\medskip

{\bf Extended Permanence Conjecture.} Any variable-$k$ weakly reversible polynomial dynamical system is permanent.

\bigskip

The global attractor conjecture is the oldest and best known of these conjectures, and has resisted efforts for a proof for over four decades, but proofs of many special cases have been obtained during this time, for example~\cite{Anderson_Shiu_2010, Anderson_2011, CNP, TDS, persistence2, ShiuSturmfels, siegel_maclean, Sontag1}.
The conjecture originates from the 1972 breakthrough work by Horn and  Jackson~\cite{Horn_Jackson}, and was formulated by Horn in 1974~\cite{Horn_1974}.

\noindent
Recently, Craciun, Nazarov and Pantea~\cite{CNP} have  proved the three-dimensional case of this conjecture, and Pantea has generalized this result for the case where the dimension of the linear invariant subspaces is at most three~\cite{persistence2}. Using a different approach, Anderson has proved the conjecture under the additional hypothesis that the graph $G$ has a single connected component~\cite{Anderson_2011}, and this result has been generalized by Gopalkrishnan, Miller, and Shiu for the case where the graph $G$ is strongly endotactic~\cite{Gopalkrishnan_Miller_Shiu_2013}. A proof of the global attractor conjecture in full generality (using as a main tool the embedding of weakly reversible polynomial dynamical systems into  toric differential inclusions, which is the main topic of this paper) has been proposed in \cite{Craciun_GAC}.

Note that all three conjectures above relate to weakly reversible polynomial dynamical systems. Indeed, it is  known that if the vertex balance condition (\ref{VBE}) is satisfied, then it follows that the E-graph $\G$  must be weakly reversible \cite{TDS, Feinberg_1979, Horn_Jackson}. 
Moreover, all these conjectures are strongly related to some version of the persistence property; in particular, it is  known that a proof of the Global Attractor Conjecture would follow if we could show that vertex balanced polynomial dynamical systems are persistent~\cite{Craciun_GAC, CNP, siegel_maclean, Sontag1}.   

In the next section we introduce toric differential inclusions, in order to facilitate the analysis of persistence properties of variable-$k$ weakly reversible polynomial dynamical systems. Indeed, we will see that the analysis of some properties of these nonautonomous systems can be reduced to the analysis of toric differential inclusions, which are not only autonomous (i.e., their right-hand sides are constant in $t$), but are also piecewise constant in $x$.

\section{Toric differential inclusions}

Given an E-graph $\G = (V,E)$, let us write $s \to s' \in E$ if $(s,s')$ is an edge of  $\G$; also let us write $s \rightleftharpoons s' \in E$ if both  $(s,s')$  and $(s',s)$  are edges of  $\G$. 

\medskip

Then, if $\G$ is reversible,  the dynamical system (\ref{polynomial_G_nonaut}) can be written as 
\begin{equation}\label{kvREV}
\frac{dx}{dt} = \sum_{s \rightleftharpoons s'\in E} \left( k_{s\to s'}(t)x^{s} - k_{s' \to s}(t)x^{s'} \right) (s' - s), 
\end{equation}
by grouping together pairs of terms given by an edge $s \to s'$ and its reverse $s' \to s$. In particular, if $\G$ consists of a {\em single} reversible edge $s\rightleftharpoons s'$, then we obtain 
\begin{equation}\label{kvREV_1}
\frac{dx}{dt} = \left( k_{s\to s'}(t)x^{s} - k_{s' \to s}(t)x^{s'} \right) (s' - s).
\end{equation}
Note that we can understand the dynamics of the system (\ref{kvREV_1}), if we think of it as a ``tug-of-war" between the forward and reverse terms, i.e., the positive and the negative monomials in (\ref{kvREV_1}). Indeed, both the forward and the reverse terms are trying to ``pull" the state $x(t)$ of the system along the same line (parallel to the vector $s' - s$), but in opposite directions. Recall that $\eps < k_e(t) < \frac{1}{\eps}$ for all $t$. Then, the domain $\RR^n_{> 0}$ can be partitioned into three regions: the region where the inequality $\eps x^{s} > \frac{1}{\eps} x^{s'}$ holds (which implies $k_{s\to s'}(t)x^{s} > k_{s'\to s}(t)x^{s'}$), the region where the inequality $\frac{1}{\eps} x^{s} > \eps x^{s'}$ holds (which implies $k_{s\to s'}(t)x^{s} > k_{s'\to s}(t)x^{s'}$), and an {\em uncertainty region} where neither one of these two inequalities are satisfied, and either one of the two terms $k_{s\to s'}(t)x^s$ and $k_{s'\to s}(t)x^{s'}$ may win the tug-of-war, or there can be a tie, due to the fact that $k_{s\to s'}(t)$ and $k_{s'\to s}(t)$ may take any values between $\eps$ and $\frac{1}{\eps}$. 

We will now show that the system (\ref{kvREV_1}) can be embedded into a piecewise constant differential inclusion defined using a partition of $\RR^n$ into three corresponding (but simpler) regions, related to the  ones above via a logarithmic transformation.

Indeed, let us define by $l_{s'-s}$ the line through the origin in $\RR^n$ generated by the vector ${s'-s}$, and by $l_{s'-s}^- \subset l_{s'-s}$ the ray starting from the origin in the direction ${s'-s}$, and by $l_{s'-s}^+ \subset l_{s'-s}$ the ray starting from the origin in the direction ${s-s'}$. Also, let us denote by $H$ the hyperplane through the origin orthogonal to ${s'-s}$. For some $\delta>0$ define the set-valued function $F_{H,\delta}$ at $X \in \RR^n$, as follows: 

%
\[
 F_{H,\delta}(X) =
  \begin{cases} 
        \hfill l_{s'-s}^+    \hfill & \text{ if  \ \ $dist(X,H) > \delta$ and $X \cdot (s'-s) > 0$} \\
      \hfill l_{s'-s}^- \hfill & \text{ if \ \  $dist(X,H) > \delta$ and $X \cdot (s'-s) < 0$} \\
      \hfill l_{s'-s}    \hfill & \text{ if \ \ $dist(X,H) \le \delta$} \\
  \end{cases}
\]

\noindent
Recall that $k_{s\to s'}(t), k_{s'\to s}(t) \in [\eps, 1/\eps]$ for all $t$. We have:

\begin{lem}\label{lem_1}
The dynamical system (\ref{kvREV_1}) (which is given by an E-graph that consists of a single reversible edge $s \rightleftharpoons s'$) is embedded in the differential inclusion 
\begin{equation}\label{tdi_1}
\frac{dx}{dt} \in F_{H,\delta}(\log x),
\end{equation}
where 
$\delta = \frac{2 |\log\eps| }{||s'-s||}$.
\end{lem}
\begin{proof}
If $dist(\log x,H) \le \delta$ there is nothing to be proved, because we already know that  the right-hand side of (\ref{kvREV_1}), i.e., the vector
$$
\left( k_{s\to s'}(t)x^{s} - k_{s' \to s}(t)x^{s'} \right) (s' - s)
$$
belongs to $l_{s'-s}$.

If $dist(\log x,H) > \delta$ and $(\log x) \cdot (s'-s) > 0$, we need to show that the right-hand side of (\ref{kvREV_1}) belongs to $l_{s'-s}^+$, i.e., we need to show that 
$$
k_{s\to s'}(t)x^{s} < k_{s' \to s}(t)x^{s'}.
$$
For this, it is sufficient to show that $\frac{1}{\eps}x^{s} < \eps x^{s'}$, which is equivalent to $x^{s'-s} > \eps^{-2}$, and, by taking logarithm on both sides of this inequality, it can also be written as 
\begin{equation}\label{133}
(\log x) \cdot (s'-s) > -2\log\eps.
\end{equation}

On the other hand, $dist(\log x,H)$  is just the dot product between the vector $\log x$ and the unit vector that is orthogonal to $H$ and on the same side of $H$ as $\log x$. But, since $(\log x) \cdot (s'-s) > 0$, this unit vector is just $\frac{s'-s}{||s'-s||}$. Then the inequality $dist(\log x,H) > \delta$ implies
\begin{equation}\label{144}
(\log x) \cdot \frac{s'-s}{||s'-s||} > \delta.
\end{equation}
Finally, given that $\delta = \frac{2 |\log\eps| }{||s'-s||}$ and $\eps\in(0,1)$, we can see that the inequalities (\ref{133}) and (\ref{144}) are equivalent. 

The case where $dist(\log x,H) > \delta$ and $(\log x) \cdot (s'-s) < 0$ is analogous, so this concludes the proof.
\end{proof}

\begin{figure}[H]
  \begin{center}
    \begin{tabular}{c}
    \hskip-1.5cm
      $(a)$ \hskip-0.1cm \includegraphics[width=1.7in]{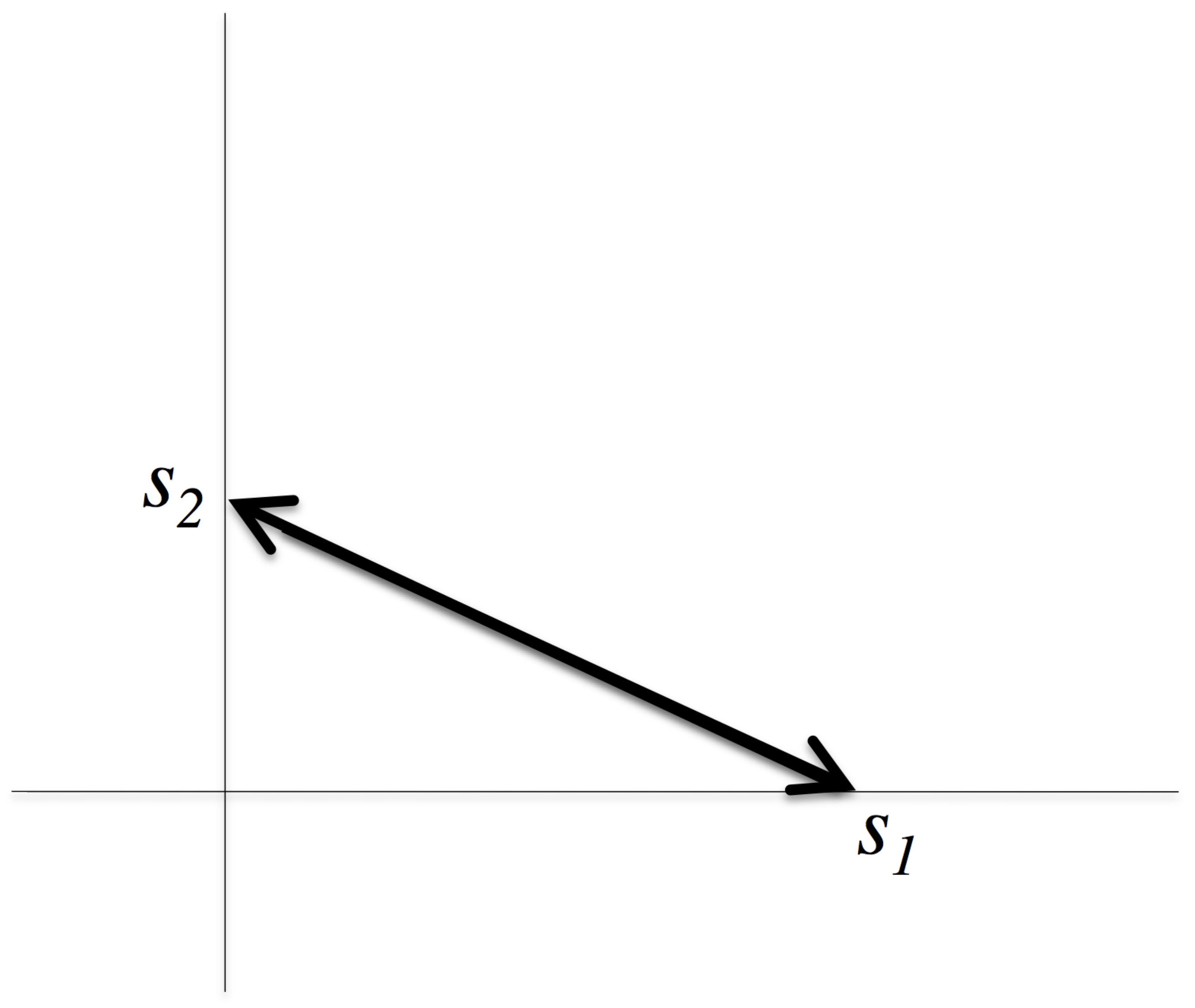} \ \ \ \ 
      $(b)$ \hskip-0.1cm  \includegraphics[width=1.7in]{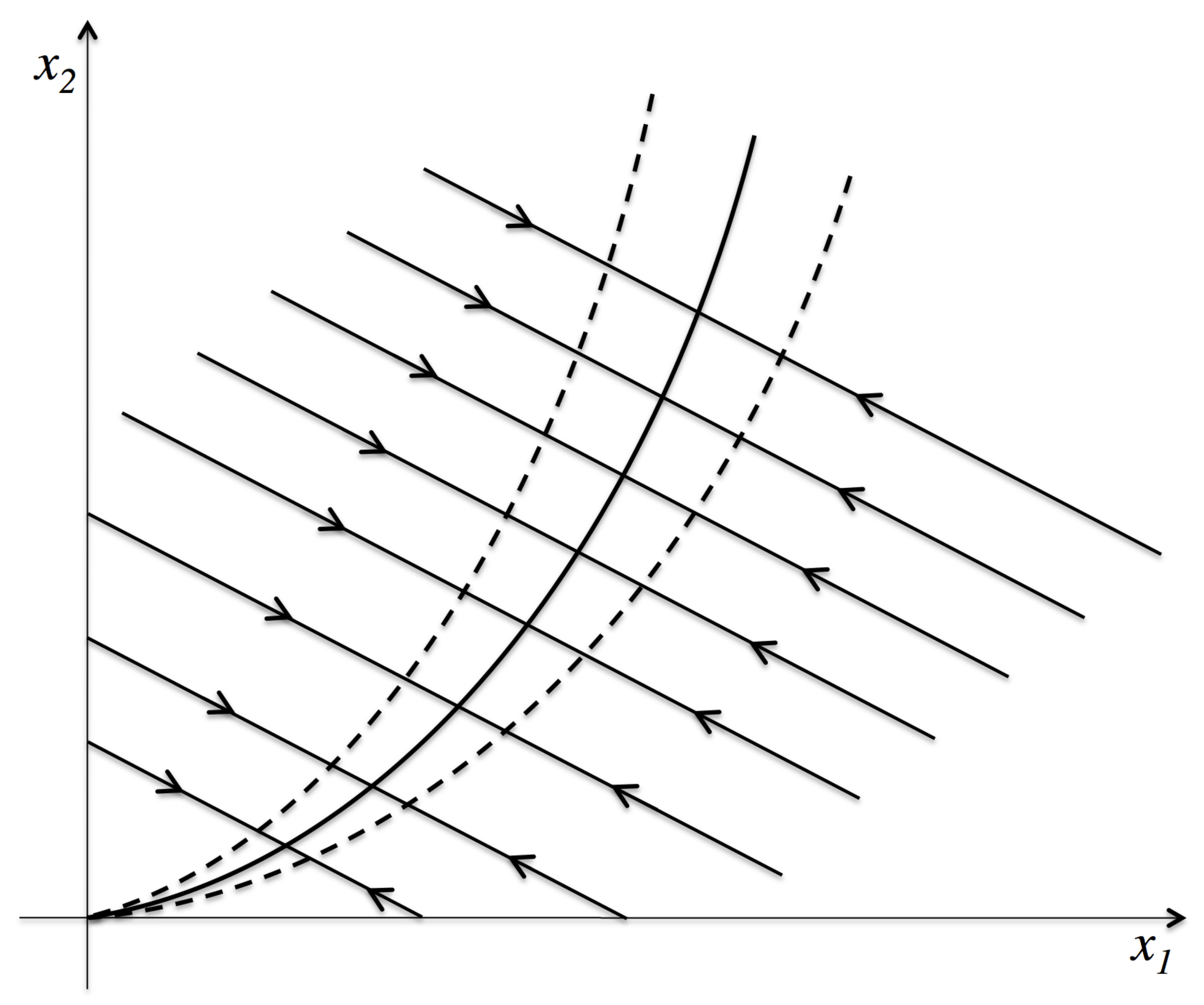} \ \ \ \ 
      $(c)$  \hskip-0.1cm \includegraphics[width=1.7in]{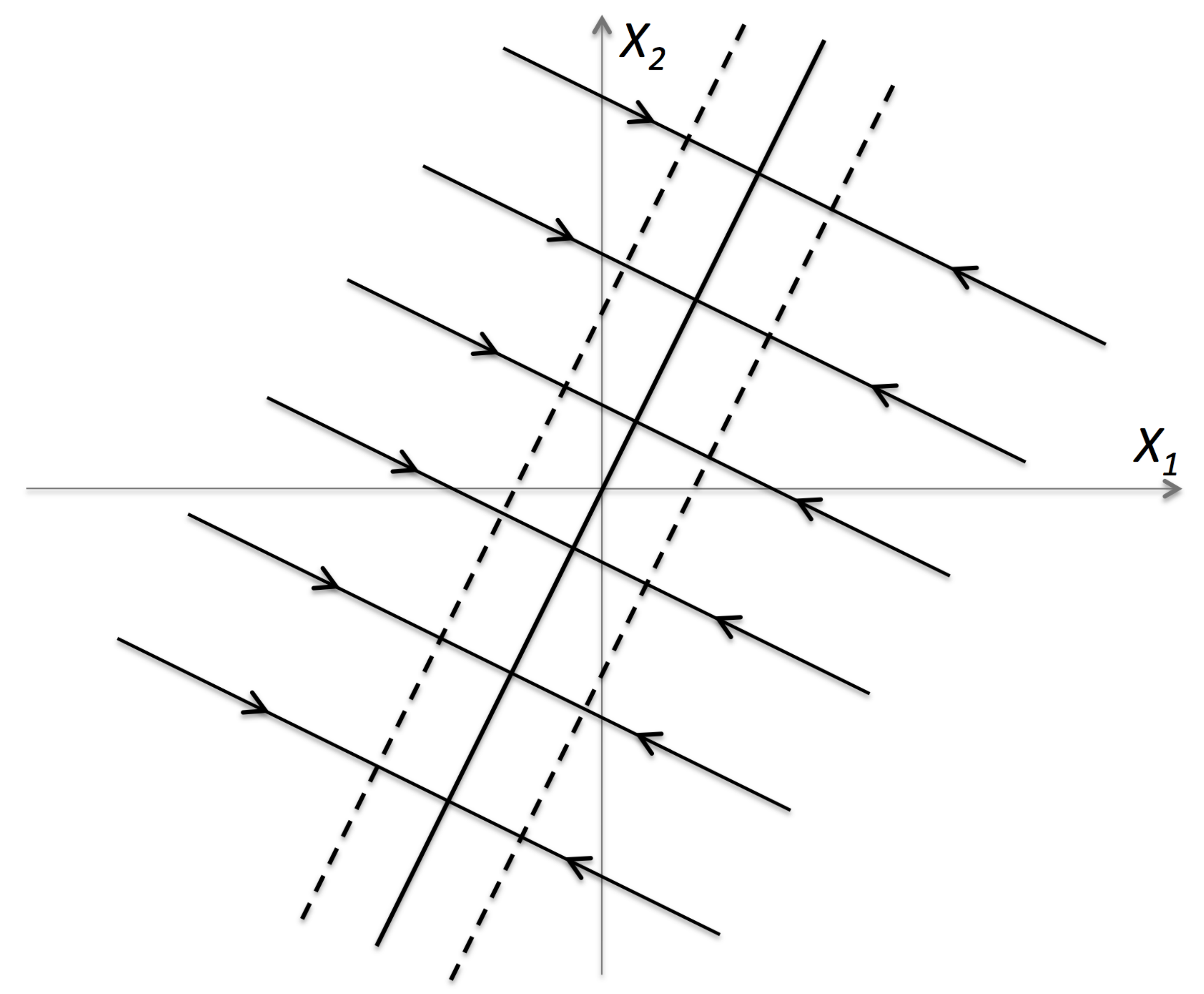}\\
      \\
    \hskip-1.5cm
      $(d)$  \hskip-0.1cm \includegraphics[width=1.7in]{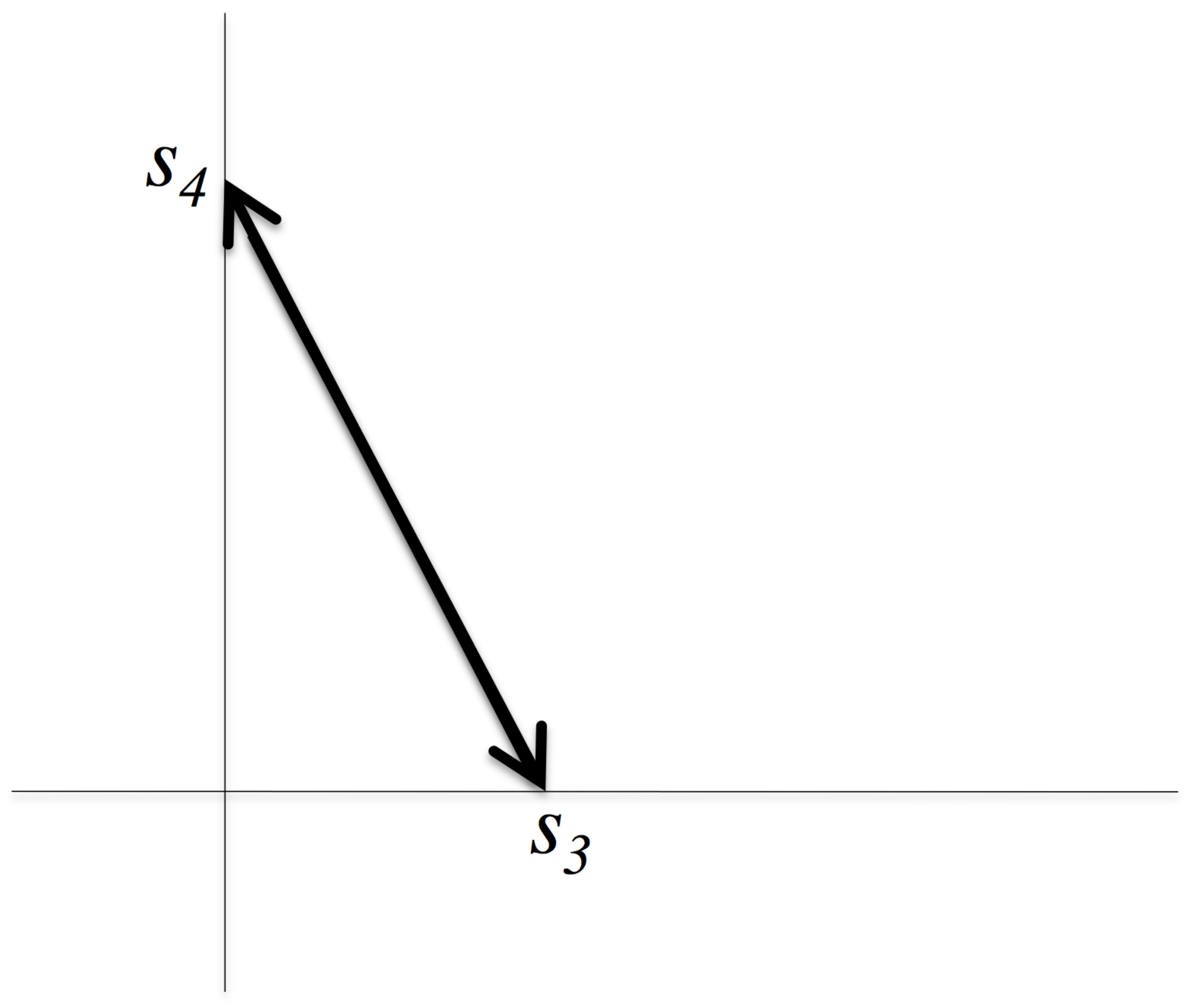} \ \ \ \ 
      $(e)$  \hskip-0.15cm \includegraphics[width=1.7in]{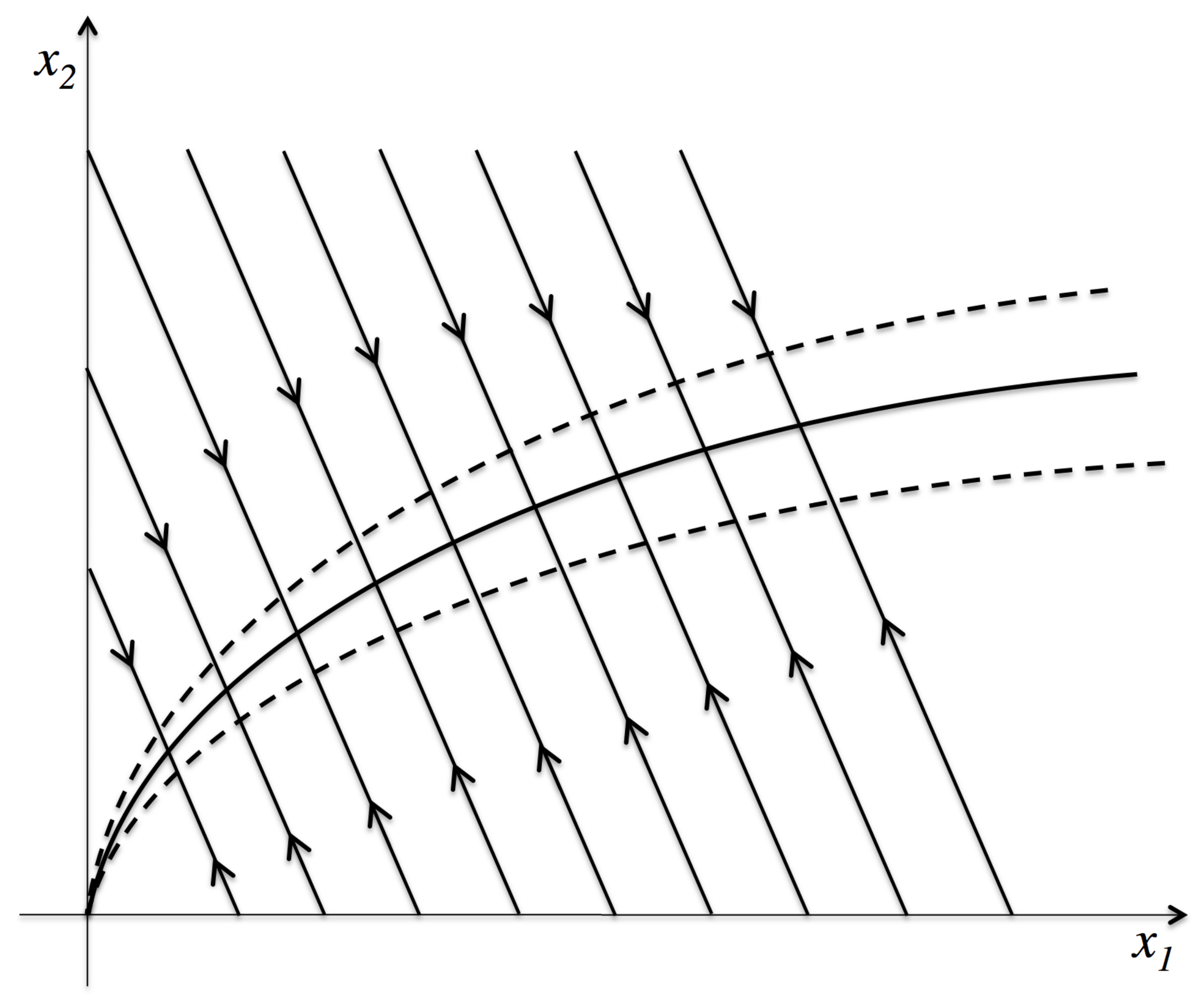} \ \ \ \ 
      $(f)$   \hskip-0.15cm \includegraphics[width=1.7in]{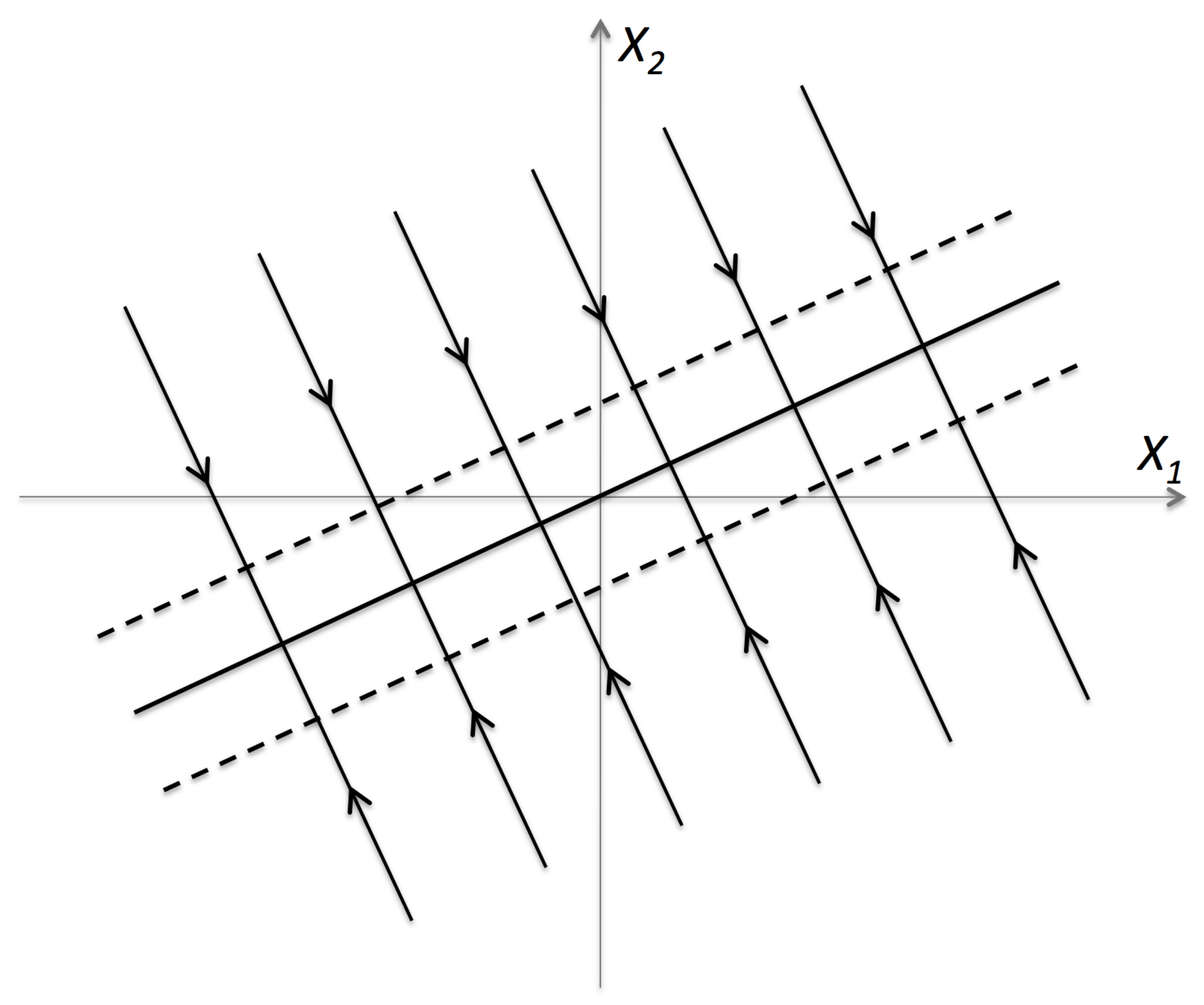}\\
      \\
    \hskip-1.5cm
      $(g)$  \hskip-0.1cm \includegraphics[width=1.7in]{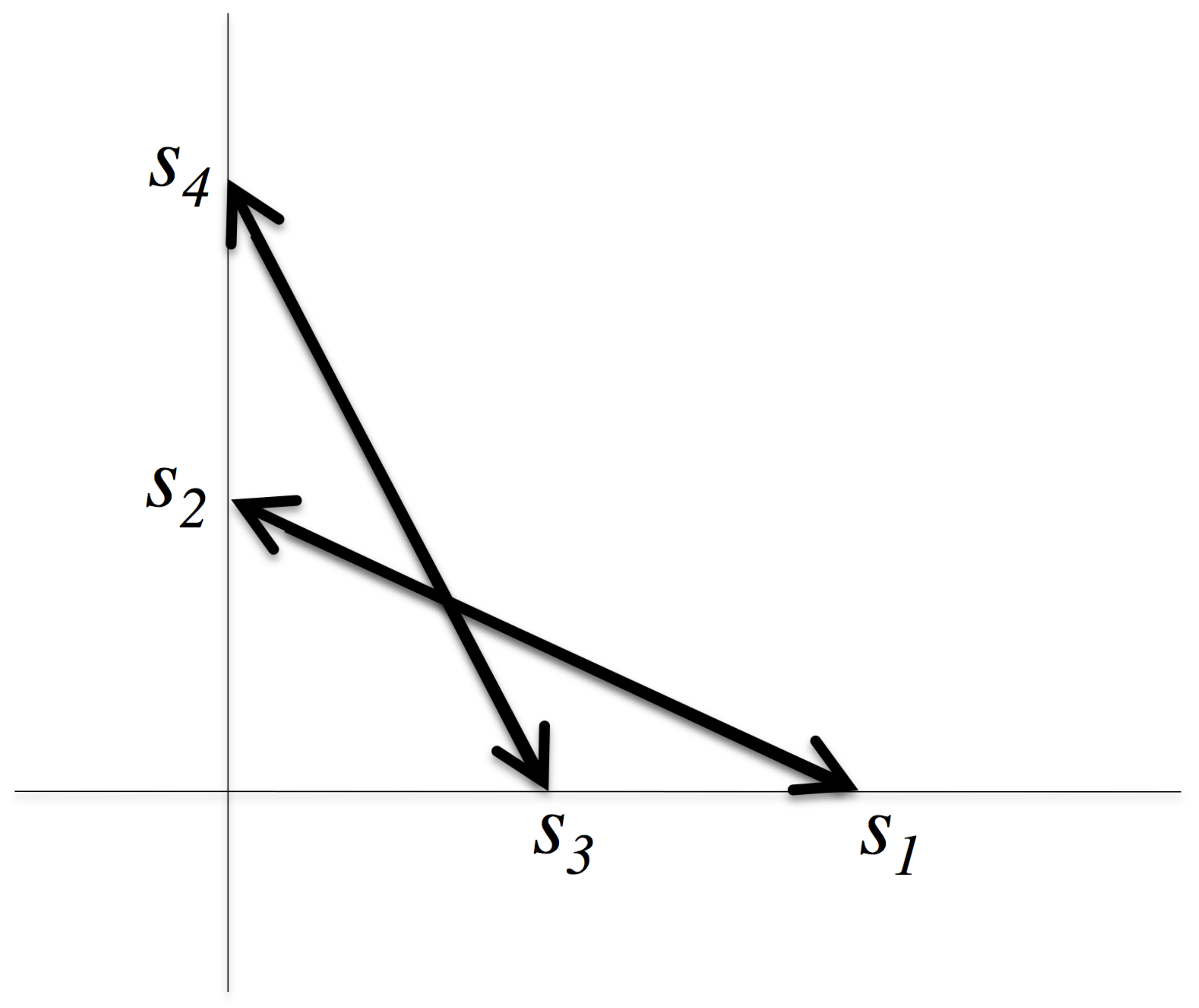} \ \ \ \ 
      $(h)$  \hskip-0.1cm \includegraphics[width=1.7in]{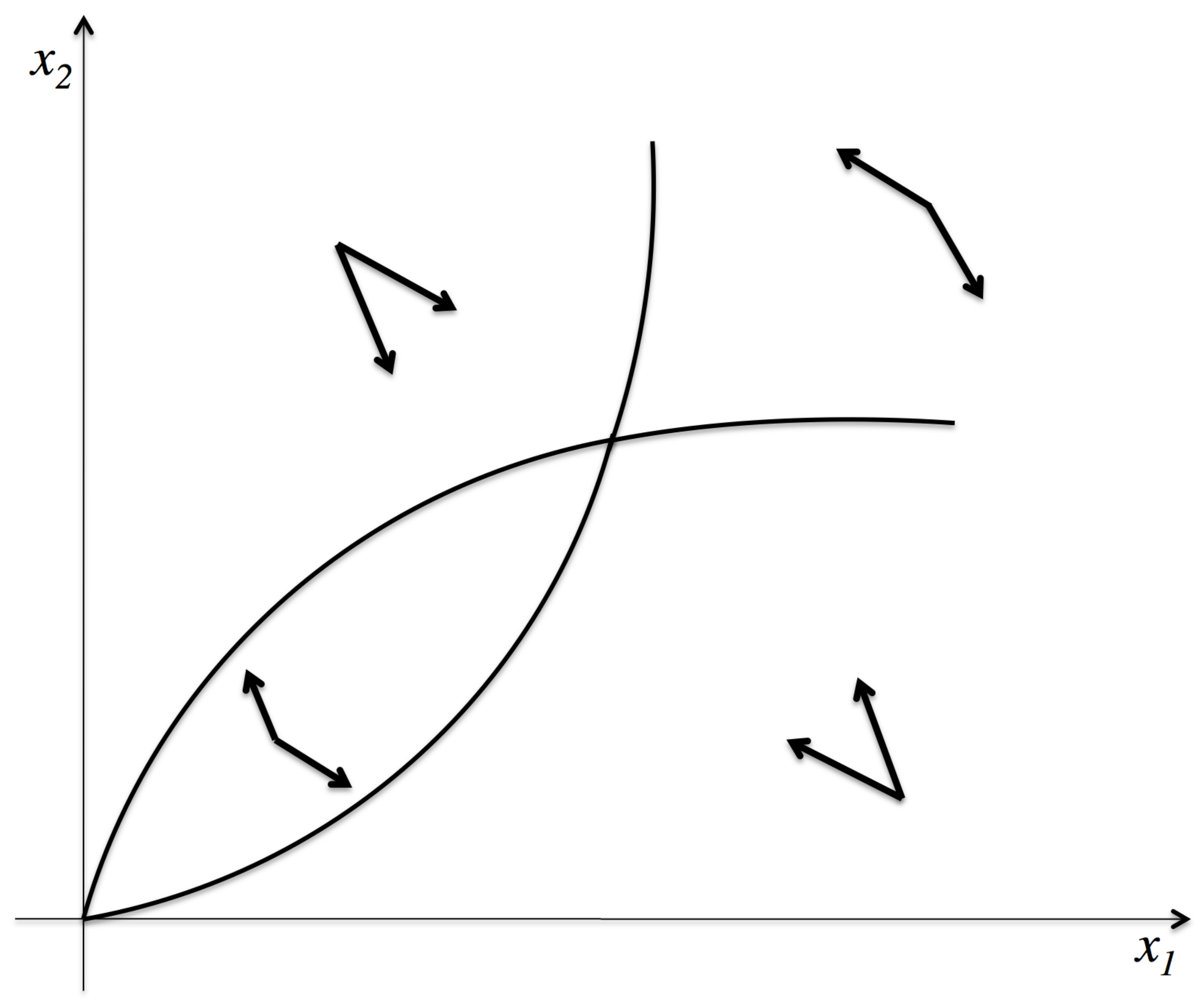} \ \ \ \ 
      $(i)$   \hskip-0.1cm \includegraphics[width=1.7in]{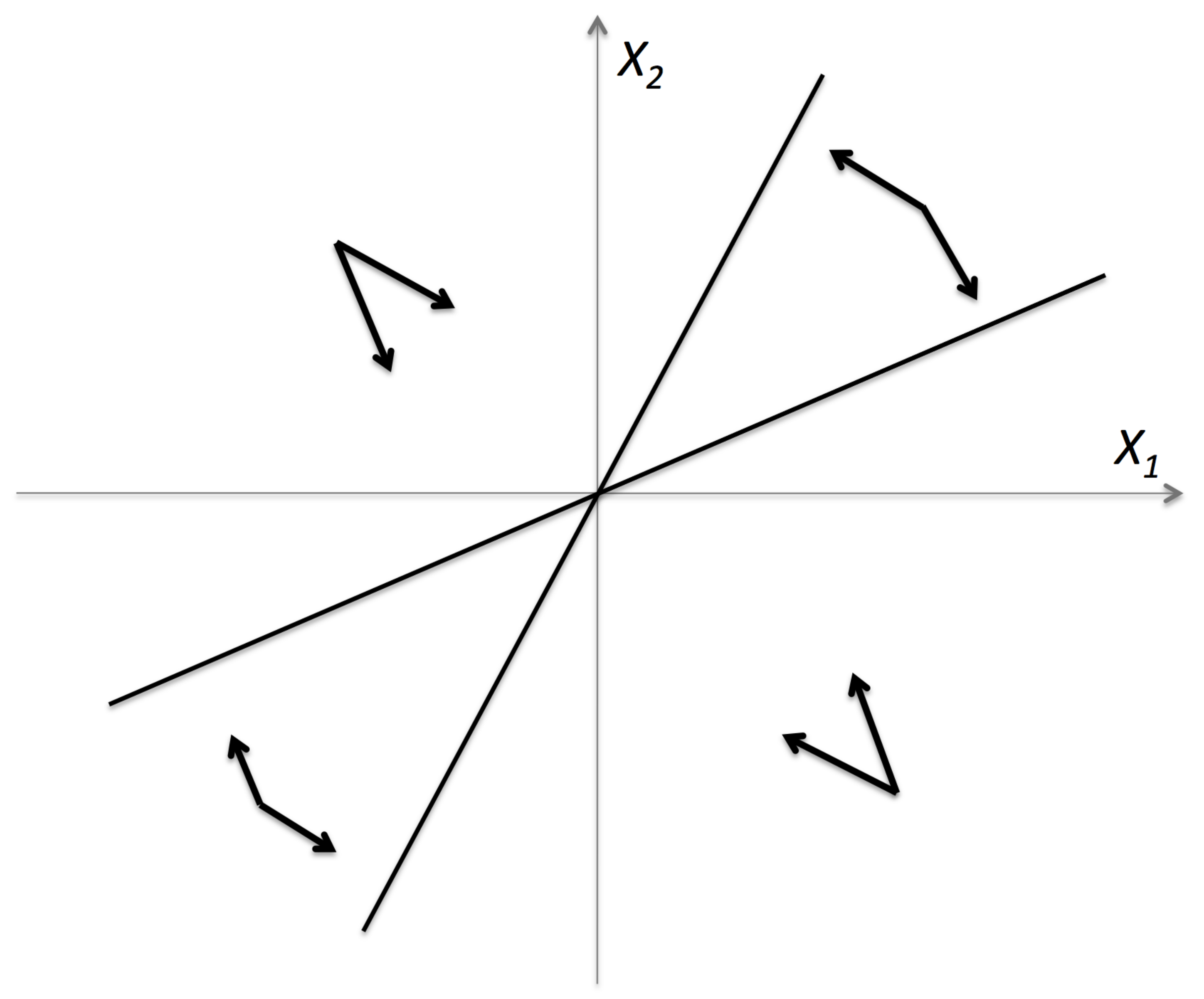}
    \end{tabular}
  \end{center}
  \caption{\label{fig:3}{\it (a) An E-graph that consists of a single reversible edge, and generates the polynomial dynamical system (\ref{example_1}). (b) The dynamics of this system in $\RR^2_{>0}$ has one-dimensional affine invariant sets, and the direction of the flow is well-defined outside an uncertainty region, which is a neighborhood of the  curve $x_1^2 = x_2$, i.e., $x^{s_1} = x^{s_2}$. (c) If we consider the diffeomorphism $X=\log x$, then the curve $x^{s_1} = x^{s_2}$ becomes the line $X\cdot (s_2-s_1) = 0$, and the uncertainty region is mapped to the set of points at distance less than some $\delta$ from this line. (d)-(f) A similar (mirror image) example. (g)-(i) Here we look at what happens if we consider an E-graph that contains {\em two} reversible edges. Note how the direction cones shown in (h) are exactly the polar cones of the cones that form the (hyperplane-generated) polyhedral fan shown in (i). \ \ See also Section 3 in~\cite{CNP} for a related example.
  }} 
\end{figure}

Therefore,  if $\G$ consists of a single reversible edge and  if $dist(\log x,H) > \delta$, then the right-hand side of the system (\ref{kvREV_1}) is a vector that is orthogonal to the hyperplane $H$ and \emph{points in the direction that goes from the point $\log x$ towards $H$} (see Fig.~\ref{fig:3}$(a)$-$(f)$ for some examples in $\RR^2$). Next, we will use this observation in order to construct a generalization of Lemma~\ref{lem_1} for all reversible E-graphs, by using the notion of \emph{polar cone}.

\bigskip

Recall that a \emph{polyhedral cone} $C\subset\RR^n$ is the set of nonnegative linear combinations of a finite set of vectors in $\RR^n$, or, equivalently, is a finite intersection of half-spaces in $\RR^n$ \cite{Rockafellar_Convex_Analysis}. For simplicity, we 
will often say \emph{cone} instead of \emph{polyhedral cone}, because the only cones we consider here are polyhedral cones.

\begin{defn}
Consider a cone $C\subset\RR^n$. The {\em polar cone} of $C$ is denoted $C^o$ and is given by
\begin{equation}\label{polar}
C^o = \{ y\in \RR^n | \ x \cdot y \le 0 \textrm{ for all } x\in C \}.
\end{equation}
\end{defn}

\noindent
The polar cone is just the negative of the better known \emph{dual cone}. Also, note that if the cone $C$ is full-dimensional (i.e., its linear span is $\RR^n$), then its polar cone $C^o$ \emph{is generated by the outer normal vectors of the codimension-1 faces of $C$}. For more information about polyhedral cones and their polar (or dual) cones see \cite{Fulton, Rockafellar_Convex_Analysis, Ziegler}. 

\bigskip

Let us now return to the  graph $\G$ that consists of a single reversible edge $s\rightleftharpoons s'$. Recall that  $H$ denotes the hyperplane through the origin and orthogonal to ${s'-s}$. Denote by $H_+$ the closed half-space of $\RR^n$ that is bounded by $H$ and contains the vector $s'-s$, and denote by $H_-$ the closed half-space of $\RR^n$ that is bounded by $H$ and contains the vector $s-s'$. Note that the polar cone $H_+^o$ is equal to the ray $l_{s'-s}^+$, and the polar cone $H_-^o$ is equal to the ray $l_{s'-s}^-$. Then the set-valued function $F_{H,\delta}$ can be rewritten as

\[
 F_{H,\delta}(X) =
  \begin{cases} 
        \hfill H_+^o                      \hfill & \text{ if \ \ $dist(X,H) > \delta$ and $X \in H_+$} \\
        \hfill H_-^o                       \hfill & \text{ if \ \  $dist(X,H) > \delta$ and $X \in H_-$} \\
        \hfill H_+^o + H_-^o    \hfill & \text{ if \ \ $dist(X,H) \le \delta$}, \\
  \end{cases}
\]
where we define $A+B = \{a+b \, | \, a\in A \textrm{ and } b\in B \}$.

\noindent
Then we can write $F_{H,\delta}(X)$ only in terms of the distance between $X$ and the half-spaces $H_+$ and $H_-$, as follows:

\[
 F_{H,\delta}(X) =
 \begin{cases} 
	\hfill H_+^o                      \hfill & \text{ if \ \ $dist(X,H_+) \le \delta$ and $dist(X,H_-) > \delta$} \\
	\hfill H_-^o                       \hfill & \text{ if \ \  $dist(X,H_-) \le \delta$ and $dist(X,H_+) > \delta$} \\
	\hfill H_+^o + H_-^o    \hfill & \text{ if \ \ $dist(X,H_+) \le \delta$ and $dist(X,H_-) \le  \delta$} \\
\end{cases}
\]

In order to be able to construct a generalization of Lemma \ref{lem_1} for the case where the E-graph $\G$ contains  {\em several} reversible edges, we want to regard the set of cones $\{ H_+, H_-, H \}$ as a special case of a {\em polyhedral fan}, because if we have several reversible edges then we have to consider a cover of $\RR^n$ given by several hyperplanes, the half-space they generate, and their intersections (see Fig. \ref{fig:3}$(g)$-$(i)$). Recall the definition of a polyhedral fan~\cite{Cox_Little_Schenck_BOOK,Fulton}:

\begin{defn}
A finite set $\F$ of polyhedral cones in $\RR^n$ is a {\em polyhedral fan} if the following two conditions are satisfied:

\medskip
$(i)$ any face of a cone in $\F$ is also in $\F$,

$(ii)$ the intersection of two cones in $\F$ is a face of both cones. 

\medskip
\noindent
We say that a polyhedral fan $\F$ is  {\em complete} if \ $\displaystyle \bigcup_{C\in\F} C = \RR^n$.
\end{defn}

\noindent
For example, given a finite set $\H$ of hyperplanes that contain the origin, consider the set $\F_\H$ of all the possible intersections of half-spaces given by hyperplanes in $\H$. Then  $\F_\H$ is a complete polyhedral fan. (Since we are only interested in complete polyhedral fans, from now on we will refer to them simply as {\em fans}.) This particular case of ``hyperplane-generated fan" will be especially relevant for motivating our  definition of toric differential inclusions.  

Indeed, we can now write $F_{H,\delta}(X)$  as
\begin{equation}\label{TDI_a}
F_{H,\delta}(X) = \sum_{\substack{C \in \{ H_+, H_-, H \} \\ dist(X,C) \le \delta}} C^o 
\end{equation}
i.e., $F_{H,\delta}(X)$ consists of all possible sums of elements of  polar cones $C^o$ such that the cone $C$ belongs to the  fan $\{H_+, H_-, H\}$, and the distance between $X$ and $C$ is at most $\delta$.
In general (see \cite{Rockafellar_Convex_Analysis}) for any cones $C_1, C_2$ we have
$$
C_1^o + C_2^o = (C_1 \cap C_2)^o
$$ 
so it follows that we can rewrite (\ref{TDI_a}) as 
\begin{equation}\label{TDI_b}
F_{H,\delta}(X) = \left(\bigcap_{\substack{C \in \{ H_+, H_-, H \} \\ dist(X,C) \le \delta}} C\right)^o 
\end{equation}
The characterizations (\ref{TDI_a}) and (\ref{TDI_b}) have the advantage that they can be easily carried over to the more general case where $\G$ consists of  {\em several} reversible edges. In that case we have several tug-of-wars going on at the same time, but for each one of them we can specify the winning direction (if any) at $x$ by calculating the distance between $X=\log x$ and some hyperplane in $\RR^n$. Depending on whether $X$ falls outside an uncertainty region or not, each reversible edge $s\rightleftharpoons s'$ of $G$ contributes one or two vectors to $F_{H,\delta}(X)$ (if one, then it is either $s'-s$ or $s-s'$, and if two, then they are $\pm(s'-s)$). 

It follows that any system (\ref{kvREV}) can be embedded into a differential inclusion on $\RR^n_+$ given by a set $\H$ of hyperplanes in $\RR^n$ and a number $\delta>0$, as follows. For each $x\in\RR^n_{> 0}$ we define $F_{\H,\delta}(\log x)$ to be the  convex cone generated by vectors orthogonal to the hyperplanes of $\H$, in the direction that goes from the point $X=\log x$ {\em towards} each hyperplane, and also the opposite direction if $X$ is at distance $<\delta$ from some hyperplane. If $X$ does not belong to any uncertainty region, then $F_{\H,\delta}(\log x)$ is defined to be exactly the {\em polar cone} $C^o$ of the (unique) cone $C\in\F_{\H}$ that contains $X$. If $X$ does belong to some uncertainty regions, then we can still describe $F_{\H,\delta}(\log x)$ in terms of polar cones, by including not just the polar of the cone of $\F_{\H}$ that contains $X$, but also the polar of each cone of $\F_{\H}$ that is at distance $\le \delta$ from $X$.

Of course, not every fan is generated by a set of hyperplanes as above. Nevertheless, we can generalize the construction described above to define a differential inclusion given by a general  fan $\F$ in $\RR^n$, as follows.
\begin{defn}\label{def_TDI}
Consider a polyhedral fan $\F$ in $\RR^n$, and a number $\delta > 0$. The {\em toric differential inclusion generated by $\F$ and $\delta$} is the differential inclusion on $\RR^n_{> 0}$ given by 
\begin{equation}\label{TDI}
\frac{dx}{dt} \in F_{\F,\delta}(\log x),
\end{equation}
where $F_{\F,\delta}$ is a set-valued function defined as
\begin{equation}\label{TDI_2}
F_{\F,\delta}(X) =  \sum_{\substack{C \in \F \\ dist(X,C) \le \delta}} C^o .
\end{equation}
\end{defn}

\noindent
In other words, the toric differential inclusion (\ref{TDI}) is a piecewise constant differential inclusion, and its right-hand side $F_{\F, \delta}(\log x)$ is the cone generated by the sum of all the polar cones $C^o$ such that $C\in\F$ and $dist(\log x,C) \le \delta$. Note that for every $x$ there is at least one such $C$, because the fan $\F$ is complete. If $C\in\F$ is a cone of dimension $< n$ and $dist(\log x,C) \le \delta$, then we say that $x$ {\em belongs to the uncertainty region of $C$}.

Like before, we can also rewrite $F_{\F,\delta}$ as 
\begin{equation}\label{TDI_3}
F_{\F,\delta}(X) =  \left( \bigcap_{\substack{C \in \F \\ dist(X,C) \le \delta}} C \right)^o .
\end{equation}

From Lemma \ref{lem_1} and the discussion above, it follows that any reversible polynomial dynamical system can be embedded into a toric differential inclusion:

\begin{prop}
\label{prop_tor_TOR}
Consider a variable-$k$  reversible polynomial dynamical system (\ref{kvREV}). Then this system can be embedded into a toric differential inclusion.
\end{prop}

Given a reversible polynomial dynamical system generated by the E-graph $\G$, the most natural such embedding is obtained if we choose $\F$ to be the fan generated by the set $\H$ of hyperplanes that are orthogonal to the edge vectors of $\G$, i.e.,
$$
\H = \{ (s'-s)^\perp | \ s \rightleftharpoons s'  \in \G \},
$$
and we choose $\delta$ as suggested by Lemma~\ref{lem_1}, i.e., $\displaystyle \delta = \max_{s \rightleftharpoons s'  \in \G}\frac{2 | \log\eps |}{|| s - s' ||}$.

\begin{figure}[H]
  \begin{center}
    \begin{tabular}{c}
      $(a)$ \includegraphics[width=2.3in]{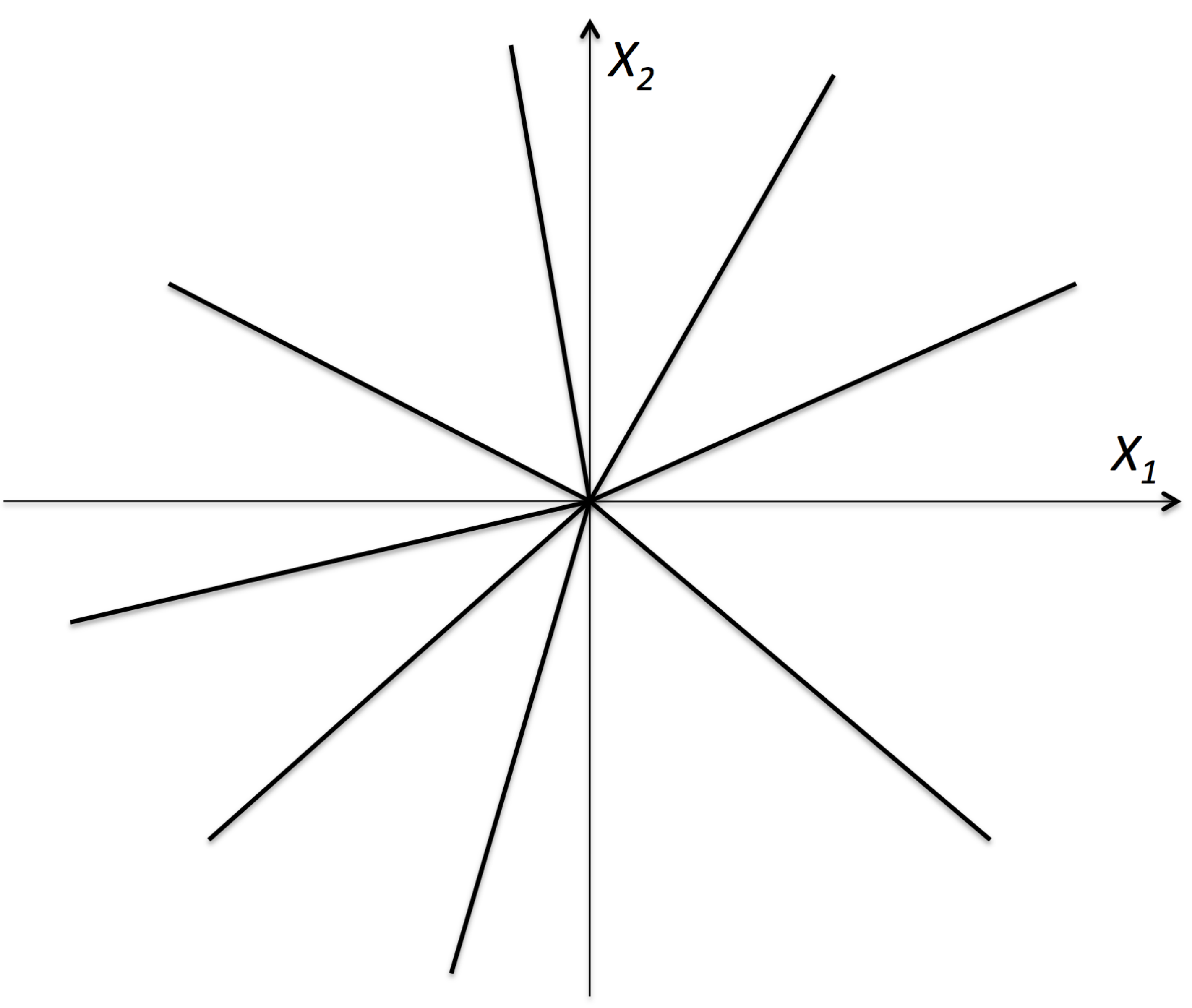} \ \ \ \ 
      $(b)$ \includegraphics[width=2.3in]{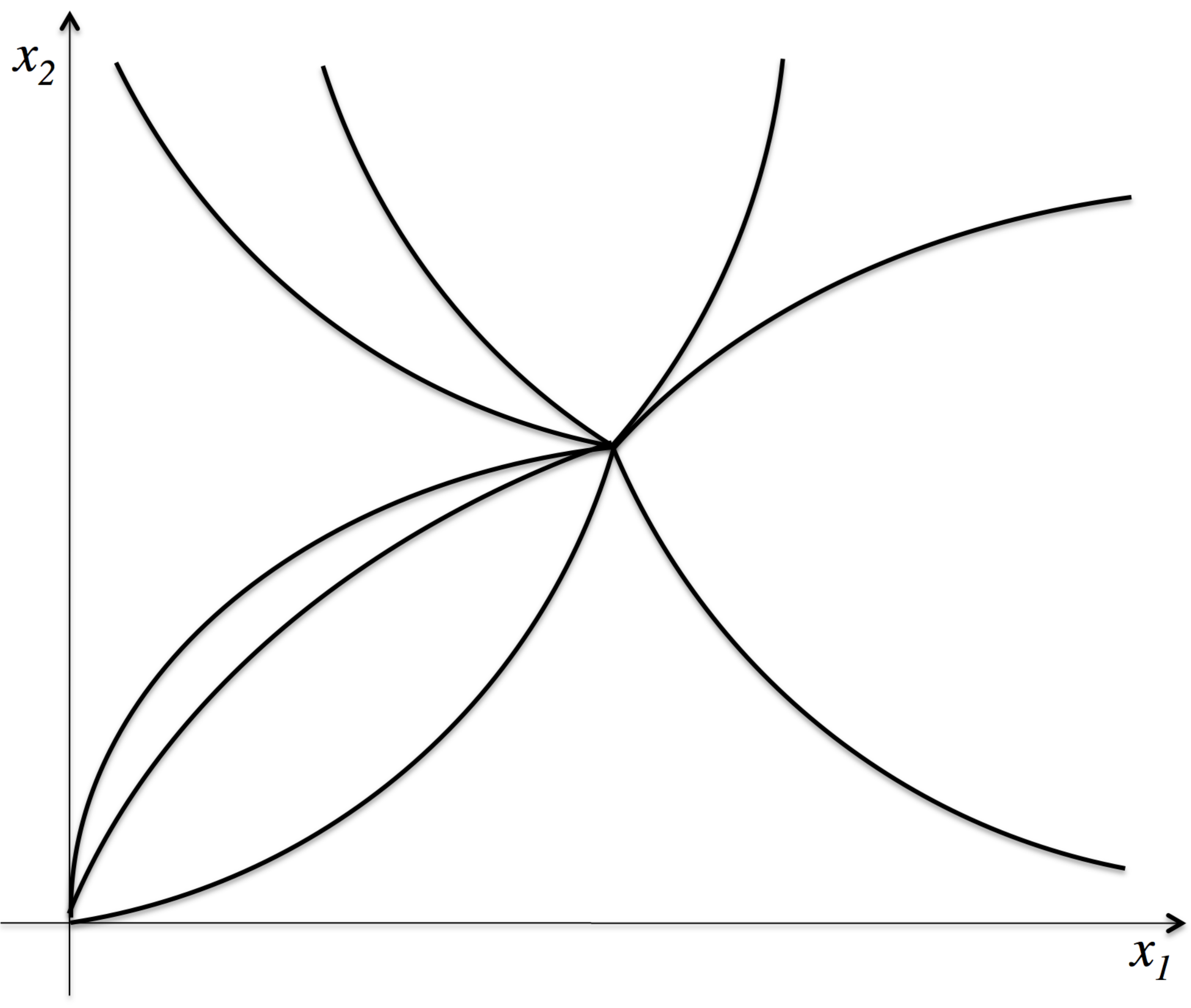}\\
      $(c)$ \includegraphics[width=2.3in]{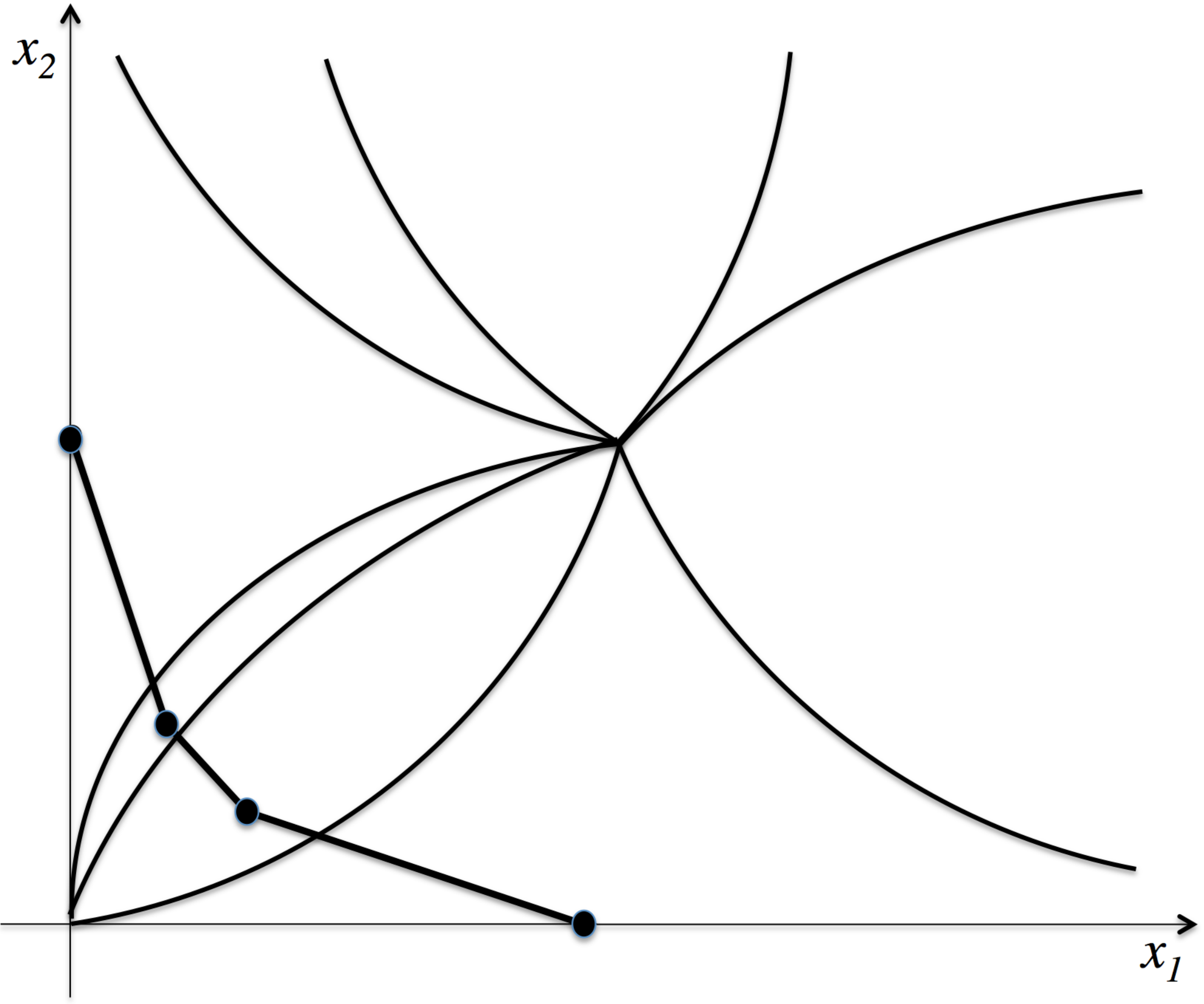} \ \ \ \ 
      $(d)$ \includegraphics[width=2.3in]{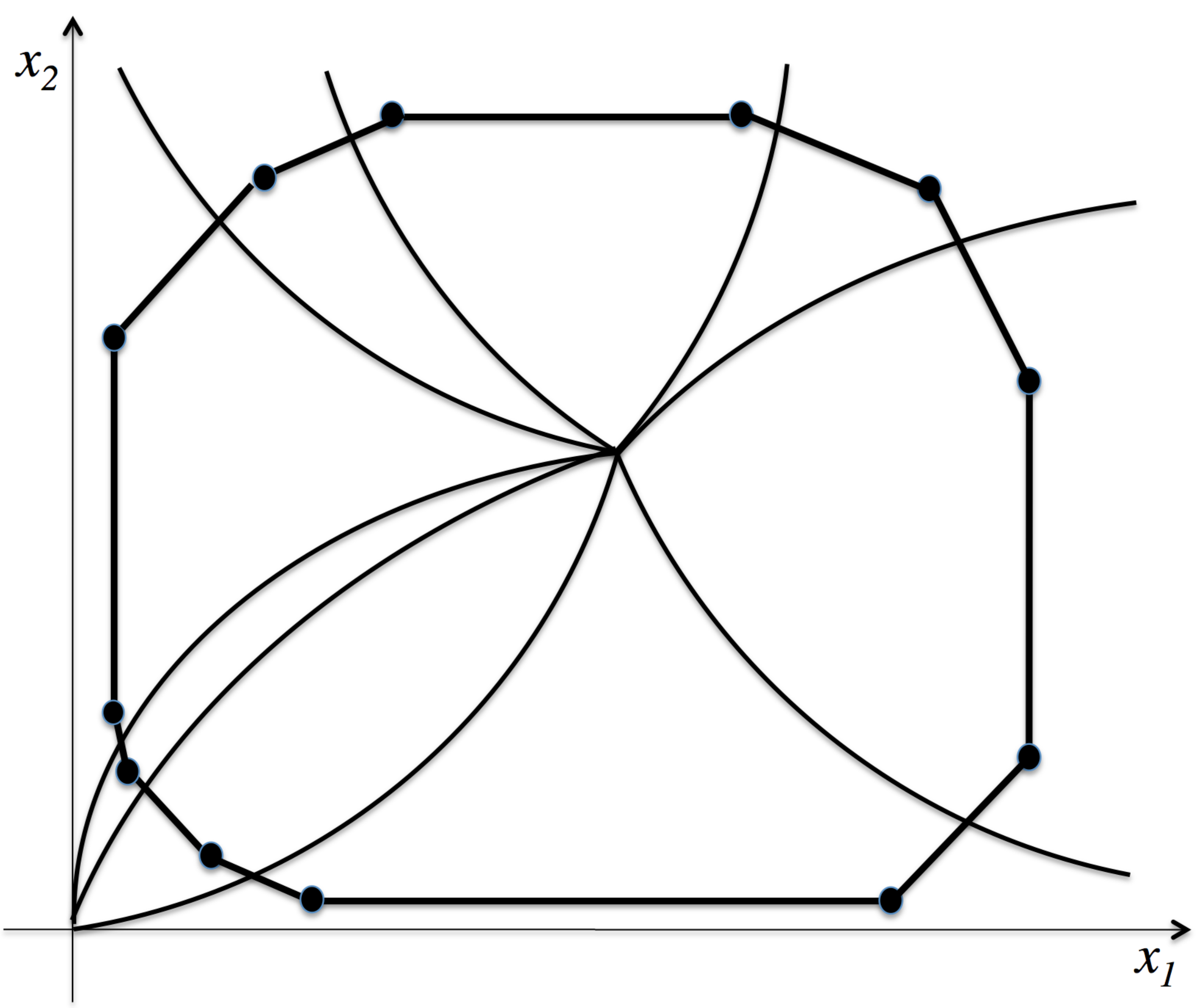}
    \end{tabular}
  \end{center}
  \caption{\label{fig:2}{\it (a) A general polyhedral fan in $\RR^2$ gives rise to a toric differential inclusion in $\RR^2_{> 0}$, whose piecewise constant domains are  sketched in~(b). We are not showing explicitly the direction cones in~(b), but they are just the polar cones of the cones in (a). Some neighborhoods of the curves shown in~(b) delimit the {\em uncertainty regions} of this toric differential inclusion. To visualize these neighborhoods in (b),  we should imagine that each curve in (b) has some nonzero thickness that represents its uncertainty region, and the right-hand side of the toric differential inclusion within that uncertainty region is a half-plane (each such half-plane is the polar cone of a ray in (a)). In (c) we see how we can use the slopes of the boundary lines of these half-planes to build a polygonal line that crosses each curve along line segments of specified slope; when such a line segment crosses a curve in (c), it must be orthogonal to the corresponding ray in (a). In (d) we see that we can follow the imposed slopes to build {\em compact invariant regions}, which allow us to prove that a polynomial dynamical system embedded in this toric differential inclusion is persistent and actually also permanent~\cite{CNP}. 
  }} 
\end{figure}

On the other hand, we are most interested in {\em weakly} reversible polynomial dynamical systems, since the conjectures we described in Section 2 refer to this  larger class of dynamical systems. In the next section we address this problem, and we prove that variable-$k$  weakly reversible polynomial dynamical systems can be embedded into toric differential inclusions. This will imply that toric dynamical systems~\cite{TDS} can be embedded into toric differential inclusions, and is some of the motivation for calling these differential inclusions ``toric".

\section{Embedding of variable-$k$ weakly reversible polynomial dynamical systems into toric differential inclusions}

As we discussed in the previous section, the simplest examples of toric differential inclusions are generated by polyhedral fans $\F_\H$ that are determined by a finite set $\H$ of hyperplanes. We will refer to this class of toric differential inclusions as {\em hyperplane-generated toric differential inclusions}. 

We have also seen in the previous section that any variable-$k$  {\em reversible} polynomial dynamical system in $\RR^n_{> 0}$ can be embedded into a (hyperplane-generated) toric differential inclusion. 
Here we show that the same is true for all variable-$k$ {\em weakly reversible} polynomial dynamical systems.

\begin{thm}
\label{thm_tor_TOR}
Consider a variable-$k$ weakly reversible polynomial dynamical system (\ref{polynomial_G_nonaut}). Then this system can be embedded into a toric differential inclusion.
\end{thm}

\begin{proof}
Denote by $\G$ a weakly reversible E-graph that generates our system. Consider first the case where $\G$ consists of a single oriented cycle. 
Then the graph $\G$ is given by $s_1 \to s_2 \to ... \to s_r \to s_1$, and the variable-$k$ weakly reversible dynamical system it generates has the form 
\begin{equation}\label{MAcycle}
\frac{dx}{dt} = \sum_{i=1}^r k_i(t) \, x^{s_i} (s_{i+1} - s_i), 
\end{equation}
where $s_{r+1} = s_1$ and $\eps \le k_i(t) \le \frac{1}{\eps}$ for some ${\eps}>0$.

Consider the set $\L$ of lines through the origin in the direction of vectors $s_i - s_j$ for all $i \neq j$, and denote by $\H$ the set of all hyperplanes that are orthogonal to a line in $\L$, i.e.,
$$
\H = \{ \ell^\perp \, | \, \ell \in \L \}.
$$
Denote by $\F_\H$ the polyhedral fan generated by the set of hyperplanes $\H$, and, for $\delta>0$, denote by $\T_{\H,\delta}$ the corresponding hyperplane-generated toric differential inclusion. We will show that there exists $\delta_0>0$ such that the single-cycle variable-$k$ weakly reversible dynamical system~(\ref{MAcycle}) is embedded in the toric differential inclusion $\T_{\H,\delta_0}$. (Note that when refer below to reversible edges $s_i\rightleftharpoons s_j$, we do {\em not} assume that these edges belong to $\G$; we only mention them as a tool in our construction of $\T_{\H,\delta_0}$.)

Choose $\delta_0>0$ large enough such that the uncertainty regions given by the {\em reversible} edges $s_i\rightleftharpoons s_j$ and $\eps$ are contained within the uncertainty regions of the toric differential inclusion $\T_{\H,\delta_0}$. For example, according to Lemma~\ref{lem_1} (see also Proposition~\ref{prop_tor_TOR}), we can choose 
\begin{equation}\label{delta_0}
\delta_0 = \max_{i \neq j}\frac{2 | \log\eps |}{|| s_i - s_j ||}.
\end{equation}

\bigskip

Let us first consider the case of a point $x\in\RR_{> 0}^n$ that does {\em not} belong to any uncertainty region of $\T_{\H,\delta_0}$. Then the point $X = \log x$ is not contained in any hyperplane in $\H$, so there must exist a cone $C$ in $\F_\H$ such that $C$ has dimension $n$ and contains the point $X$ in its interior.  
%
%
%
%
We will  show that the right-hand side of~(\ref{MAcycle}) is contained in the polar cone $C^o$. 


Consider a vector $w$ in the interior of $C$, and consider the orthogonal projections of the vectors  $s_1, s_2 ,..., s_r$ on the line $\ell_w$ that passes through the origin in the direction given by $w$. Then no two such projections are the same, because $w$ does not belong to any of the hyperplanes in $\H$. In other words, we have that $s_i \cdot w \neq s_j \cdot w$ whenever $s_i  \neq s_j$.
We now use these projections to give {\em a second set of names} to the vectors $s_1, s_2 ,..., s_r$, say $v_1, v_2 ,..., v_r$, to record the order in which these projections appear along the line $\ell_w$. More precisely, we choose the names $v_1, v_2 ,..., v_r$ in the 
order (from largest to smallest) of the values of $v_l \cdot w$,  so for example,  $v_1$ equals the $s_i$ that has the largest value of $s_i \cdot w$, $v_2$ equals the $s_i$ that has the second-largest value of $s_i \cdot w$, and so on. 
Note also that, since the signs of the dot products 
$
(v_{l+1}-v_l) \cdot w
$ 
cannot change as $w$ is allowed to vary in the interior of $C$, it follows that 
the new names $v_1, v_2 ,..., v_r$ do {\em not} depend on the particular choice of vector $w$ in the interior of $C$. In other words, the dot products 
$
(v_{l+1}-v_l) \cdot w
$
are all {\em negative} numbers, for all $w$ in the interior of $C$.
Therefore, the vectors $v_2-v_1, v_3-v_2, ..., v_r-v_{r-1}$ belong to $C^o$. 

So, in order to show that the right-hand side of (\ref{MAcycle}) is included in $C^o$, it is enough to show that it can be written as a positive linear combination of the vectors $v_2-v_1, v_3-v_2, ..., v_r-v_{r-1}$. 

If $s_{i_1} \!= v_1, \ s_{i_2} \!= v_2,\  ...,\ s_{i_r} \!= v_r$, then note that $(i_1,i_2,...,i_r)$ is a permutation of $(1,2,...,r)$. If we denote the inverse permutation by $(j_1,j_2,...,j_r)$, it follows that $s_1 = v_{j_1}$, $s_2 = v_{j_2}$, and so on.

Then we have $s_2 - s_1 = v_{j_2} - v_{j_1}$. If $j_2 > j_1$ we write
$$
s_2 - s_1 = \sum_{l=j_1}^{j_2-1} (v_{l+1} - v_l),
$$
and if $j_2 < j_1$ we write 
$$
s_2 - s_1 = -\sum_{l=j_2}^{j_1-1} (v_{l+1} - v_l),
$$
We do the same for $s_3 - s_2, \ s_4 - s_3$, and so on. 
This way, we write each difference $s_{i+1} - s_i$ from the right-hand side of~(\ref{MAcycle}) in terms of the vectors $\pm(v_{l+1} - v_l)$, with $l=1,2, ...,r-1$. 
Therefore we can re-group terms to obtain
\begin{equation}\label{MAcycle_new}
\frac{dx}{dt} = \sum_{l=1}^{r-1} \Phi_l (v_{l+1} - v_l), 
\end{equation}
where $\Phi_l$ is a sum of several terms of the form $k_i \, x^{v_i}$, with various signs. 

Note now that the positive terms inside $\Phi_l$ correspond to edges of the form $v_{m} \to v_{n}$ with $m \le l < n$, and negative terms inside $\Phi_l$ correspond to edges of the form $v_{m} \to v_{n}$ with $n \le l < m$. This means that the positive terms inside $\Phi_l$ are a sum of terms of the form $k_i(t) \, x^{v_i}$ with $i \le l$, and the negative terms inside $\Phi_l$ are a sum of terms of the form $k_i(t) \, x^{v_i}$ with $i > l$. 
In particular, since $\G$ is a cycle, it follows that $\Phi_l$ contains at least one positive term, and at least one negative term.
Recall that $X = \log x$ is in the interior of $C$. Then the dot products of the form  $(v_{l+1} - v_l) \cdot X$ must be negative for all $l$, which implies that 
 $x^{v_l} >x^{v_{l+1}}$ for all $l$. Moreover, since $x$ does not belong to any uncertainty region, and due to our choice of $\delta_0$~(see (\ref{delta_0}) and Lemma~\ref{lem_1}), this inequality remains the same even if we include the terms $k_l(t)$ and $k_{l+1}(t)$, and we obtain $k_l(t)x^{v_l} >k_{l+1}(t)x^{v_{l+1}}$. Therefore, we have
$$
k_1(t) x^{v_1} > k_2(t) x^{v_2} > ... > k_r(t) x^{v_r}.
$$

Also, note that the number of positive terms inside $\Phi_l$ is the same as the number of negative terms inside $\Phi_l$, because the graph $\G$ is a cycle. 
Therefore, the sum of the positive terms inside $\Phi_l$ dominates the sum of the negative terms inside $\Phi_l$, for each $l$. In conclusion, the right-hand side of~(\ref{MAcycle_new}) (and therefore the right-hand side of~(\ref{MAcycle})) is a positive linear combination of the vectors $v_{l+1} - v_l$, for $1 \le l \le r-1$, so it belongs to $C^o$.

\bigskip

Consider now the case where $x$ {\em does} belong to one or more uncertainty regions of the toric differential inclusion $\T_{\H,\delta_0}$. Recall that the cone of $\T_{\H,\delta_0}$ at $x$ is $F_{\H,\delta_0} (X)$ (see Definition~\ref{def_TDI}).
Then, by using (\ref{TDI_2}) and the calculations in the proof of Lemma \ref{lem_1}, we conclude that the set  $F_{\H,\delta_0} (X)$ is a cone generated by two types of vectors: vectors of the ``first type", which are of the form $\pm(s_i - s_j)$ such that $s_i\rightleftharpoons s_j$ is a reversible edge and $x$ is in the uncertainty region of the hyperplane orthogonal to $s_i - s_j$, and vectors of the ``second type",  which are of the form $s_l - s_m$ such that $s_l\rightleftharpoons s_m$ is a reversible edge and  $x$ is {\em not} in the uncertainty region of the hyperplane orthogonal to  $s_l - s_m$, and moreover $(s_l - s_m) \cdot X < 0$.
Note that, without loss of generality we can assume that {\em not} all vectors are of the first type; otherwise we immediately obtain that the right-hand side of~(\ref{MAcycle})    belongs to $F_{\H,\delta_0} (X)$.

Consider now a vector $w$ in the interior of the polar cone of the cone $F_{\H,\delta_0} (X)$; then $w$ satisfies $(s_i - s_j) \cdot w = 0$ for vectors of first type, and $(s_l - s_m) \cdot w < 0$  for vectors of second type. 
%
%
We also consider  the orthogonal projections of the vectors  $s_1, s_2 ,..., s_r$ on the line $\ell_w$ that passes through the origin in the direction given by $w$. Unlike the previous case, in this case some projections will coincide; more precisely, if $s_i$ and $s_j$ are like in the first type above, then their projections will coincide because $s_i \cdot w = s_j \cdot w$.

Nevertheless, we can still give a second set of names to the vectors $s_1, s_2 ,..., s_r$, say $v_1, v_2 ,..., v_r$, to record the order in which these projections appear along the line $\ell_w$, with the caveat that we will have one or more cases where the projections coincide. Our ordering is chosen such that  
\begin{equation}\label{1234}
(v_{l+1}-v_l) \cdot w \le 0.
\end{equation}
Note that if we allow $w$ to vary within the interior of the polar cone of the cone $F_{\H,\delta_0} (X)$, the inequality (\ref{1234}) will still hold. This implies that whenever the projections on $\ell_w$ of  $v_l$ and  $v_{l+1}$ are distinct, the vector $v_{l+1}-v_l$ belongs to the cone $F_{\H,\delta_0} (X)$; moreover, if their projections coincide, then both vectors $\pm(v_{l+1}-v_l)$ belong to $F_{\H,\delta_0} (X)$. 

Now we proceed exactly like in the previous case. We can still re-group terms like in formula (\ref{MAcycle_new}), and, in order to conclude that the right-hand side of (\ref{MAcycle_new}) belongs to $F_{\H,\delta_0} (X)$, we only have to check that $\Phi_l > 0$ for values of $l$ where projections of  $v_l$ and  $v_{l+1}$ are distinct. But, in the same way as before, such $\Phi_l$ have an equal number of positive and negative terms of the form $k_i x^{v_i}$, and, also as before, the positive terms are larger than the negative terms.

%

\bigskip

Finally, if the weakly reversible graph $\G$ is {\it not} a single oriented cycle, then we write it as a union of cyclic graphs, $\displaystyle \G = \bigcup_{i=1}^g \G_i$, and we can argue as above for each such $\G_i$. We obtain that the variable-$k$ weakly reversible dynamical systems given by the cycle $\G_i$ are embedded in toric differential inclusions generated by some set of hyperplanes $\H_i$.
Note now that the right-hand side of a variable-$k$ weakly reversible dynamical system given by $\G$ can be decomposed into a sum of terms, such that each term is of the form given by the right-hand side of a variable-$k$ weakly reversible dynamical system determined by $\G_i$.  (We may have to use smaller $\eps_i$ values for the terms in the decomposition, because the same edge of $\G$ may belong to several graphs $\G_i$.) 
\noindent
Note also that if we define the set of hyperplanes  
$$
\displaystyle \H = \bigcup_{i=1}^g \H_i
$$
then, for any $i$ and $\delta$, the toric differential inclusion given by the fan $\F_{\H_i}$ and  $\delta$ {\em is embedded in} the toric differential inclusion given by the fan $\F_{\H}$ and  $\delta$, because every cone $C_i \in \F_{\H_i}$ can be written as a union of cones from $\F_{\H}$, and whenever $C \subset \tilde C$ it follows that $C^o \supset \tilde C^o$.
Then we conclude that any variable-$k$ mass-action system given by $\G$ can be embedded into a toric differential inclusion generated by the set of hyperplanes~$\displaystyle \H$.
\end{proof}

\bigskip

\section{Applications and conclusions}

In this paper we have introduced toric differential inclusions, and we have shown that any polynomial dynamical system on the positive orthant is generated by an E-graph (which is not unique). Moreover, if this E-graph can be chosen to be weakly reversible, then the polynomial dynamical system can be embedded into a toric differential inclusion. Most importantly, toric differential inclusions have a rich geometric structure that can be used in the construction of invariant regions needed for the proof of important conjectures in this field~\cite{Craciun_GAC, CNP}. 

The idea of thinking about  polynomial dynamical systems as being generated by  E-graphs was  inspired by  the way  reaction networks generate polynomial dynamical systems under the assumption of mass-action kinetics. Indeed, the set of polynomial dynamical systems generated by  E-graphs $\G = (V,E)$ that satisfy $V \subset \ZZ_{\ge 0}^n$ is exactly the same as the set of all mass-action dynamical systems~\cite{ABCM_2018, Brunner_2018}.  

Note also that while the formulations of the conjectures in Section 2.2 are more general than the usual formulations for mass-action systems (which restrict the exponent vectors to be nonnegative), there is a simple way to show that these two versions are actually equivalent, by ``time-rescaling" via multiplication by a scalar field of the form $x^{(M\cdot{\mathbf 1)}}$, where  $\mathbf 1$ is a vector with all 1 components. This multiplication can be chosen such that it shifts the E-graph into the non-negative orthant, while preserving  all trajectory curves~\cite{Craciun_GAC}.

On the other hand, our results on embedding weakly reversible polynomial dynamical systems into toric differential inclusions suggests other kinds of {\em generalizations of the Persistence Conjecture and of the Permanence Conjecture}, as follows. Note that, in the proof of Lemma~\ref{lem_1}, in order to be able to obtain that embedding into a differential inclusion, it is not really necessary to know that  the values of $k_{s \to s'}(t),\ k_{s' \to s}(t)$ are contained in an interval of the form $[\eps, \frac{1}{\eps}]$; exactly the same calculations will work if we just know that 
\begin{equation}\label{ratio}
\eps_0 \le \frac{k_{s \to s'}(t)}{k_{s' \to s}(t)} \le \frac{1}{\eps_0}
\end{equation}
for some $\eps_0 > 0$.
Similarly, in Proposition~\ref{prop_tor_TOR}, we don't need to assume that all the time-dependent parameter values $k_{s \to s'}(t)$ are bounded away from zero and infinity; we just need to assume that for any reversible reaction the ratio of the corresponding parameter values is bounded away from zero and infinity, as in (\ref{ratio}).

Even in Theorem~\ref{thm_tor_TOR}, in the case of a weakly reversible E-graph $\G$, it is enough to assume that we have 
\begin{equation}\label{ratio_2}
\eps_0 \le \frac{k_{s \to s'}(t)}{k_{\tilde s \to \tilde s'}(t)} \le \frac{1}{\eps_0}.
\end{equation}
for any edges $s \to s$ and $\tilde s \to \tilde s'$ that are {\em in the same connected component of $\G$}. 

These observations suggest that, for example, if we could take advantage of the embedding into toric differential inclusions and prove this stronger version of the Persistence Conjecture, we would obtain the following interesting conclusion for the dynamics of chemical and biochemical reaction networks:  if several (weakly reversible) reaction networks or pathways are coupled together, then the resulting dynamics will still be persistent even if the external factors that influence each pathway are widely different in size.   For example, if we have two weakly reversible biochemical pathways, and the reaction rate parameters in one pathway are modulated by temperature, while in the other pathway they are modulated by a signaling protein whose concentration is unrelated to temperature, then by coupling together these two pathways we should still maintain the persistence property.

\bigskip

Also,  generally speaking, the embedding of weakly reversible mass-action systems into toric differential inclusions provides a more rigorous interpretation for the standard ``hand-waving" intuition behind the Persistence and Permanence Conjectures. Namely, for mass-action systems it is quite reasonable to think that, if every reaction is part of a cycle   (i.e., if the reaction network is weakly reversible) then the chemical reactions should not be able to drive any concentration to zero, because if a reaction consumes a species, then soon a chain of reactions that follow it will work to replenish  that species. This intuition is a bit too vague to be turned into a proof, but, as we can see in Fig.~\ref{fig:3}$(h)$, we now have a more concrete way to think about it: if we focus on some  region where the right-hand side of the toric differential inclusion is constant, then this constant cone of directions seems to point ``toward the middle" of the positive quadrant, i.e., away from the boundary, and also away from infinity.

\bigskip

In particular, given an embedding of a two-variable polynomial dynamical system into a toric differential inclusion, we can immediately obtain families of invariant regions for this dynamical system in $\RR^2_{>0}$, as illustrated in Fig.~\ref{fig:2}. Such an embedding allows us to construct ``zero-separating curves" (as shown in Fig.~\ref{fig:2}$(c)$) which prevent positive trajectories from approaching the origin. As shown in~\cite{CNP}, when constructed for variable-$k$ polynomial dynamical systems, these curves are the key tool for a proof of the persistence of vertex-balanced dynamical systems in $\RR^3_{>0}$. Similarly, given any positive initial point, we can use the embedding to construct a compact invariant region that contains that point (as shown in Fig.~\ref{fig:2}$(d)$).

Therefore, the embeddings of some polynomial dynamical systems into toric differential inclusions allow us to give very short proofs of the main results in~\cite{CNP} (i.e., a proof of the Persistence Conjecture in $\RR^2_{>0}$, and of the Global Attractor Conjecture in $\RR^3_{>0}$) and also to generalize the Persistence Conjecture result in $\RR^2_{>0}$ as we described above; see~\cite{Craciun_GAC} for more details. More importantly, similar geometric constructions of invariant regions based on toric differential inclusions can also be done in higher dimensions~\cite{Craciun_GAC}.

\bigskip

Outside the setting considered here, global convergence results for mass-action systems have been recently used to study  reaction-diffusion equations via the methods of lines~\cite{method_of_lines}, have been adapted for analyzing some versions of discrete Boltzmann equations~\cite{Craciun_Tran}, and an entropy method has been used to study a large class of reaction-diffusion systems that arise from vertex-balanced networks~\cite{Desvillettes_2017}. Interestingly, reversibility and weak reversibility play a role in these works as well.

\section{Acknowledgments} 

This work benefited from feedback from the organizers and participants of the Workshop on the Global Attractor Conjecture held at San Jose State University in March 2016. Exceptionally useful and detailed reviewer comments helped improve the presentation of these results. We also acknowledge  support from the National Science Foundation under grants DMS-1412643 and DMS-1816238.

\end{document}